\documentclass[11pt, authoryear, times]{elsarticle}       
\journal{Arxiv}              
\pdfoutput=1
\usepackage{ulem}
\usepackage{amssymb,amsfonts,amsmath,amsthm,bigints,bm}
\usepackage[onehalfspacing]{setspace}
\usepackage[top=1.0in, bottom=1.0in, left=1.0in, right=1.0 in]{geometry}
\usepackage[bottom]{footmisc}
\usepackage{indentfirst}
\usepackage[hyperindex,breaklinks]{hyperref}
\usepackage{cleveref}
\usepackage{enumitem}
\usepackage{epsfig,pgf}
\usepackage{float}
\usepackage{multirow}
\usepackage{rotating,longtable,booktabs, colortbl,siunitx}       
\allowdisplaybreaks
\usepackage[authoryear]{natbib}
\usepackage{subcaption,float}

\usepackage{dcolumn}
\newcolumntype{d}[1]{D{.}{.}{#1}}
\let\oldref\ref
\renewcommand{\ref}[1]{(\oldref{#1})}

\newtheorem{theorem}{Theorem}
\newtheorem{lemma}{Lemma}
\newtheorem{corollary}[theorem]{Corollary}

\newdefinition{definition}{Definition}

\numberwithin{equation}{section}

\def\N{\mathbb N}

\def\F{\mathcal F}
\def\d{\mathrm d}
\def\R{\mathbb R}
\newcommand{\norm}[1]{\left\lVert#1\right\rVert}
\newcommand{\abs}[1]{\left\lvert#1\right\rvert}


\begin{document}

\begin{frontmatter}
\title{\textbf{Quantile Regression Neural Networks: A Bayesian Approach}
}

\author[msu]{S.R.~Jantre}
\ead{jantresa@msu.edu}
\author[msu]{S.~Bhattacharya}
\ead{bhatta61@msu.edu}
\author[msu]{T.~Maiti\corref{cor1}}
\ead{maiti@msu.edu}

\cortext[cor1]{Corresponding author}

\address[msu]{Department of Statistics and Probability, Michigan State University, MI, USA  48824}


\begin{abstract}
This article introduces a Bayesian neural network estimation method for quantile regression assuming an asymmetric Laplace distribution (ALD) for the response variable. It is shown that the posterior distribution for feedforward neural network quantile regression is asymptotically consistent under a misspecified ALD model. This consistency proof embeds the problem from density estimation domain and uses bounds on the bracketing entropy to derive the posterior consistency over Hellinger neighborhoods. This consistency result is shown in the setting where the number of hidden nodes grow with the sample size. The Bayesian implementation utilizes the normal-exponential mixture representation of the ALD density. The algorithm uses Markov chain Monte Carlo (MCMC) simulation technique - Gibbs sampling coupled with Metropolis-Hastings algorithm. We have addressed the issue of complexity associated with the afore-mentioned MCMC implementation in the context of chain convergence, choice of starting values, and step sizes.  We have illustrated the proposed method with simulation studies and real data examples. 
\end{abstract}

\end{frontmatter}




\section{Introduction} 

\label{Section1} 


Quantile regression (QR), proposed by \citet{Koenker-Basset-1978}, models  conditional quantiles of the dependent variable as a function of the covariates. The method supplements the least squares regression and provides a more comprehensive picture of the entire conditional distribution. This is particularly useful when the relationships in the lower and upper tail areas are of greater interest. Quantile regression has been extensively used in wide array of fields such as economics, finance, climatology, and medical sciences, among others \citep{Koenker-2005}. Quantile regression estimation requires specialized algorithms and reliable estimation techniques which are available in both frequentist and Bayesian literature. Frequentist techniques include simplex algorithm \citep{Dantzig-1963} and the interior point algorithm \citep{Karmarkar-1984}, whereas Bayesian technique using Markov chain Monte Carlo (MCMC) sampling was first proposed by \citet{Yu-Moyeed-2001}. Their approach employed the asymmetric Laplace distribution (ALD) for the response variable, which connects to frequentist quantile estimate, since its maximum likelihood estimates are equivalent to the quantile regression using check-loss function \citep{Koenker-Machado-1999}.  Recently, \citet{Kozumi-Kobayashi-2011} proposed a Gibbs sampling algorithm, where they exploit the normal-exponential mixture representation of the asymmetric Laplace distribution which considerably simplified the computation for Bayesian quantile regression models.

Artificial neural networks are helpful in estimating possibly non-linear models without specifying an exact functional form. The neural networks which are most widely used in engineering applications are the single hidden-layer feedforward neural networks. These networks consist of a set of inputs $\bm{X}$, which are connected to each of $k$ hidden  nodes, which, in turn, are connected to an output layer $(\bm{O})$. In a typical single layer feedforward neural network, the outputs are computed as 
\begin{equation*}
\bm{O}_i=b_0 + \sum_{j=1}^k b_j \psi \left( c_{j0} + \sum_{h=1}^{p} \bm{X}_{ih} c_{jh} \right)
\end{equation*}
where, $c_{jh}$ is the  weight from input $\bm{X}_{ih}$ to the hidden node $j$. Similarly, $b_j$ is the weight associated with the hidden unit $j$. The $c_{j0}$ and $b_0$ are the biases for the hidden nodes and the output unit. The function $\psi(.)$ is a nonlinear activation function. Some common choices of $\psi(.)$ are the sigmoid and the hyperbolic tangent functions. The interest in neural networks is motivated by the universal approximation capability of feedforward neural networks (FNNs) \citep{Cybenko-1989,Funahashi-1989,Hornik-et-al-1989}. According to these authors, standard feedforward neural networks with as few as one hidden layer whose output functions are sigmoid functions are capable of approximating any continuous function to a desired level of accuracy, if the number of hidden layer nodes are sufficiently large. \citet{Taylor-2000} introduced a practical implementation of quantile regression neural networks (QRNN) to combine the approximation ability of neural networks with robust nature of quantile regression. Several variants of QRNN have been developed such as composite QRNN where neural networks are extended to the linear composite quantile regression \citep{Xu-et-al-2017} and later \citet{Cannon-2018} introduced monotone composite QRNN which guaranteed the non-crossing of regression quantiles. 

Bayesian neural network learning models find the predictive distributions for the target values in a new test case given the inputs for that case as well as inputs and targets for the training cases. Early work of \citet{Buntine-Weigend-1991} and \citet{Mackay-1992} has inspired widespread research in Bayesian neural network models. Their work implemented Bayesian learning using Gaussian approximations. Later, \citet{Neal-1996} applied Hamiltonian Monte Carlo in Bayesian statistical applications. Further, Sequential Monte Carlo techniques applied to neural network architectures are described in \citet{DeFreitas-2001}. A detailed review of MCMC algorithms applied to neural networks is presented by \citet{Titterington-2004}. Although Bayesian neural networks have been widely developed in the context of mean regression models, there has been limited or no work available on its development in connection to quantile regression both from a theoretical and implementation standpoint. We also note that the literature on MCMC methods applied to neural networks is somewhat limited due to several challenges including lack of parameter identifiability, high computational costs, and convergence failures \citep{Papamarkou-et-al-2019}. 

In this article, we bridge the gap between Bayesian neural network learning and quantile regression modeling. The proposed Bayesian quantile regression neural network (BQRNN) uses a single hidden layer FNN with sigmoid activation function, and a linear output unit. The Bayesian procedure has been implemented using Gibbs sampling combined with random walk Metropolis-Hastings algorithm. Further the posterior consistency of our method has been established under the framework of \citet{Lee-2000} and \citet{Sriram-et-al-2013}. The former has shown posterior consistency in the case of Bayesian neural network for mean models while the later have shown it in the case of Bayesian quantile regression (BQR). We follow the frameworks of \citet{Lee-2000} to prove the consistency which are built using the results presented in \citet{Funahashi-1989}, \citet{Hornik-et-al-1989} and others. Their approach borrows many of the ideas for establishing consistency in the context of density estimation from \citet{Barron-et-al-1999}. We handle the case of ALD responses with the help of \citet{Sriram-et-al-2013}'s handling of ALD in BQR scenario. 

The rest of this article is organized as follows. Section \ref{Section2} introduces quantile regression and its Bayesian formulation by establishing relationship between quantile regression and asymmetric Laplace distribution. In Section \ref{Section3}, we discuss Bayesian quantile regression neural network model and the specific priors used in this study. Further, we propose a hierarchical BQRNN model and introduce a MCMC procedure which couples Gibbs sampling with random walk Metropolis-Hastings algorithm. We conclude this section with an overview of the posterior consistency results for our model. Section \ref{Section4} presents Monte Carlo simulation studies and real world applications. A brief discussion and conclusion is provided in Section \ref{Section5}. Detailed proofs of theorems and their corresponding lemmas are presented in the Appendix.  
\section{Bayesian Quantile Regression} 

\label{Section2} 


Quantile regression offers a practically important alternative to mean regression by allowing the inference about the conditional distribution of the response variable through modeling of its conditional quantiles. Let $Y$ and $\bm{X}$ denote the response and the predictors respectively and $\tau \in (0,1)$ be the quantile level of the conditional distribution of $Y$ and $F(.)$ be the cumulative distribution function of $Y$, then a linear conditional quantile function of $Y$ is denoted as follows
\begin{equation*} 
Q_{\tau}(y_i |\bm{X_i} = \bm{x_i}) \equiv F^{-1}(\tau) = \bm{x_i}^T {\beta}(\tau), \hspace{1cm} i = 1,\dots,n,
\end{equation*} 
where ${\beta}(\tau) \in \mathbb{R}^{p}$ is a vector of quantile specific regression coefficients of length $p$. The aim of quantile regression is to estimate the conditional quantile function $Q(.)$.

Let us consider the following linear model in order to formally define the quantile regression problem,
\begin{equation} \label{quant_lin_model}
Y = \bm{X}^T \beta(\tau) + \varepsilon, 
\end{equation}
where $\varepsilon$ is the error vector restricted to have its $\tau^\text{th}$ quantile to be zero, i.e. $\int _{-\infty} ^ 0 f({\varepsilon}_i) d{\varepsilon}_i = \tau$. The probability density of this error is often left unspecified in the classical literature. The estimation through quantile regression proceeds by minimizing the following objective function
\begin{equation} \label{obj_fun}
\min_{{\beta}(\tau) \in \mathbb{R}^p} \sum_{i=1}^n \rho_{\tau} (y_i - \bm{x_i}^T \beta(\tau)) 
\end{equation}
where $\rho_{\tau}(.)$ is the check function or quantile loss function with the following form:
\begin{equation} \label{check_fun}
\rho_{\tau}(u) = u . \{\tau - I(u<0)\},
\end{equation}
$I(.)$ is an indicator function. 
Classical methods employ linear programming techniques such as the simplex algorithm, the interior point algorithm, or the smoothing algorithm  to obtain quantile regression estimates for ${\beta}({\tau})$ \citep{Madsen-Nielsen-1993,Chen-2007}. The statistical programming language R makes use of \href{https://cran.r-project.org/web/packages/quantreg/quantreg.pdf}{\texttt{quantreg}} package \citep{Koenker-2017} to implement quantile regression techniques whilst confidence intervals are obtained via bootstrap \citep{Koenker-2005}. 

Median regression in Bayesian setting has been considered by \citet{Walker-Mallick-1999} and \citet{Kottas-Gelfand-2001}. In quantile regression, a link between maximum-likelihood theory and minimization of the sum of check functions, \autoref{obj_fun}, is provided by asymmetric Laplace distribution (ALD) \citep{Koenker-Machado-1999,Yu-Moyeed-2001}. This distribution has location parameter $\mu$, scale parameter $\sigma$ and skewness parameter $\tau$. Further details regarding the properties of this distribution are specified in \citet{Yu-Zhang-2005}. If $Y {\sim}$ ALD$(\mu,\sigma,p)$, then its probability distribution function is given by
\begin{equation*}
f(y|\mu,\sigma,\tau) = \frac{\tau(1-\tau)}{\sigma} \exp \left \{- \rho_{\tau} \left (\frac{y-\mu}{\sigma}\right ) \right \} 
\end{equation*}
As discussed in \citet{Yu-Moyeed-2001}, using the above skewed distribution for errors provides a way to implement Bayesian quantile regression effectively. According to them, any reasonable choice of prior, even an improper prior, generates a posterior distribution for $\beta(\tau)$. Subsequently, they made use of a random walk Metropolis Hastings algorithm with a Gaussian proposal density centered at the current parameter value to generate samples from analytically intractable posterior distribution of $\beta(\tau)$. 

In the aforementioned approach, the acceptance probability depends on the choice of the value of $\tau$, hence the fine tuning of parameters like proposal step size is necessary to obtain the appropriate acceptance rates for each $\tau$. \citet{Kozumi-Kobayashi-2011} overcame this limitation and showed that Gibbs sampling can be incorporated with AL density being represented as a mixture of normal and exponential distributions. Consider the linear model from \autoref{quant_lin_model}, where $\varepsilon_i \sim$ ALD$(0,\sigma,\tau)$, then this model can be written as
\begin{equation} \label{quant_lin_ald_model}
y_i = \bm{x_i}^T {\beta}(\tau) + \theta v_i + \kappa \hspace{0.1cm}\sqrt[]{\sigma v_i} u_i, \hspace{1cm} i = 1,\dots,n,
\end{equation}
where, $u_i$ and $v_i$ are mutually independent, with $u_i \sim $N$(0,1), v_i \sim \mathcal{E}(1/\sigma)$ and $\mathcal{E}(1/\sigma)$ is the exponential distribution with mean $\sigma$. The $\theta$ and $\kappa$ constants in  \autoref{quant_lin_ald_model} are given by 
\begin{equation*}
\theta = \frac{1-2\tau}{\tau(1-\tau)} \quad and \quad \kappa = \sqrt[]{\frac{2}{\tau(1-\tau)}}
\end{equation*}
Consequently, a Gibbs sampling algorithm based on normal distribution can be implemented effectively. Currently, \href{https://cran.r-project.org/web/packages/Brq/Brq.pdf}{\texttt{Brq}} \citep{Alhamzawi-2018} and \href{https://cran.r-project.org/web/packages/bayesQR/bayesQR.pdf}{\texttt{bayesQR}} \citep{Benoit-et-al-2017} packages in R provide Gibbs sampler for Bayesian quantile regression. We are employing the same technique to derive Gibbs sampling steps for all except hidden layer node weight parameters for our Bayesian quantile regression neural network model. 
 
\section{Bayesian Quantile Regression Neural Network} 

\label{Section3} 


\subsection{Model}
In this paper we focus on feedforward neural networks with a single hidden layer of units with logistic activation functions, and a linear output unit. Consider the univariate response variable $Y_i$ and the covariate vector $\bm{X_i}$ $(i = 1,2,\dots,n)$. Further, denote the number of covariates by $p$ and the number of hidden nodes by $k$ which is allowed to vary as a function of $n$. Denote the input weights by $\gamma_{jh}$ and the output weights by $\beta_j$. Let, $\tau \in (0,1)$ be the quantile level of the conditional distribution of $Y_i$ given $\bm{X_i}$ and keep it fixed.  Then, the resulting conditional quantile function is denoted as follows  
\begin{equation} \label{quant_fun}
Q_{\tau}(y_i |\bm{X_i} = \bm{x_i}) = \beta_0 + \sum_{j=1}^k \beta_j \frac{1}{1 + \exp{ \left( -\gamma_{j0}- \sum_{h=1}^p \gamma_{jh} x_{ih} \right)}} = \beta_0 + \sum_{j=1}^k \beta_j \psi(\bm{x_{i}^T\gamma_j}) = \bm{\beta}^T \bm{\eta_i(\gamma)} = \bm{L_i} \bm{\beta}
\end{equation} 
where, $\bm{\beta} = (\beta_0,..,\beta_k)^T$, $\bm{x_i} = (1,x_{i1},..,x_{ip})^T$, $\bm{\eta_i(\gamma)} = (1,\psi(\bm{x_{i}^T\gamma_1}),..,\psi(\bm{x_{i}^T\gamma_k}))^T$ and $\bm{L} = (\bm{\eta_1(\gamma)},..,\bm{\eta_n(\gamma)})^T$, $i = 1,..,n$. $\psi(.)$ is the logistic activation function.

The specified model for $Y_i$ conditional on $\bm{X_i}=\bm{x_i}$ is given by $Y_i \sim ALD(\bm{L_i} \bm{\beta},\sigma,\tau)$ with a likelihood proportional to 
\begin{equation}\label{ALD_likelihood}
    \sigma^{-n} \exp \left \{-\sum_{i=1}^n \frac{|\varepsilon_i| + (2\tau-1)\varepsilon_i}{2\sigma} \right \}
\end{equation}
where, $\varepsilon_i = y_i - \bm{L_i\beta}$. The above ALD based likelihood can be represented as a location-scale mixture of normals  \citep{Kozumi-Kobayashi-2011}. For any $a,b > 0$, we have the following equality \citep{Andrews-Mallows-1974}
\begin{equation*}
    \exp\{-|ab|\} = \int_0^{\infty} \frac{a}{\sqrt{2\pi v}}\exp \left\{-\frac{1}{2} (a^2 v + b^2 v^{-1})\right\} \d v
\end{equation*}
Letting $a = 1/\sqrt{2\sigma}$, $b=\varepsilon/\sqrt{2\sigma}$, and multiplying by $\exp\{-(2\tau-1)\varepsilon/2\sigma\}$ the \autoref{ALD_likelihood} becomes

\begin{equation} \label{ALD_mixture_rep}
    \sigma^{-n} \exp \left \{-\sum_{i=1}^n \frac{|\varepsilon_i| + (2\tau-1)\varepsilon_i}{2\sigma} \right \} = \prod_{i=1}^n \int_0^{\infty} \frac{1}{\sigma \sqrt{4\pi\sigma v_i}} \exp \left\{ - \frac{(\varepsilon_i - \xi v_i)^2}{4\sigma v_i} - \zeta v_i \right\} \d v_i
\end{equation}
where, $\xi = (1-2\tau)$ and $\zeta = \tau(1-\tau)/\sigma$. \autoref{ALD_mixture_rep} is beneficial in the sense that there is no need to worry about the prior distribution of $v_i$ as it is extracted in the same equation. The exact prior of $v_i$ in \autoref{ALD_mixture_rep} is exponential distribution with mean $\zeta^{-1}$ and it depends on the value of $\tau$.

Further we observe that, the output of the afore-mentioned neural network remains unchanged under a set of transformations, like certain weight permutations and sign flips which renders the neural network non-identifiable. For example, in the above model \eqref{quant_fun}, take $p,k=2$ and $\beta_0, \gamma_{j0} =0$. Then,
\begin{equation*}
\sum_{j=1}^2 \beta_j \psi(\bm{x_{i}^T\gamma_j}) = \beta_1 \left[ 1 + \exp{ \left( -\gamma_{11}x_{i1} -\gamma_{12} x_{i2} \right)}\right]^{-1} + \beta_2 \left[ 1 + \exp{ \left( -\gamma_{21}x_{i1} -\gamma_{22} x_{i2} \right)}\right]^{-1}
\end{equation*} 
In the foregoing equation we can notice that when $\beta_1=\beta_2$ , two different sets of values of $(\gamma_{11},\gamma_{12},\gamma_{21},\gamma_{22})$ obtained by flipping the signs, namely (1,2,-1,-2) and (-1,-2,1,2) result in the same value for $\sum_{j=1}^2 \beta_j \psi(\bm{x_{i}^T\gamma_j})$. However, as a special case of lemma 1 of \citet{Ghosh-et-al-2000}, the joint posterior of the parameters is proper if the joint prior is proper, even in the case of posterior invariance under the parameter transformations. Note that, as long as the interest is on prediction rather parameter estimation, this property is sufficient for predictive model building. In this article we focus only on proper priors, hence the non-identifiability of the parameters in \eqref{quant_fun} doesn't cause any problem. 

\subsection{MCMC Algorithm}
We take mutually independent priors for $\bm{\beta},\bm{\gamma_1},\dots,\bm{\gamma_k}$ with $\bm{\beta}\sim \N(\bm{\beta_0},\sigma_0^2 \bm{I}_{k+1})$ and $\bm{\gamma_j} \sim \N(\bm{\gamma_{j0}},\sigma_1^2 \bm{I}_{p+1})$, $j=1,\dots,k$. Further, we take inverse gamma prior for $\sigma$ such that $\sigma \sim \mathbb{IG}(a/2,b/2)$. Prior selection is problem specific and it is useful to elicit the chosen prior from the historical knowledge. However, for most practical applications, such information is not readily available. Furthermore, neural networks are commonly applied to big data for which a priori knowledge regarding the data as well as about the neural network parameters is not typically known. Hence, prior elicitation from experts in the area is not applicable to neural networks in practice. As a consequence, it seems reasonable to use near-diffuse priors for the parameters of the given model. 

Now, the joint posterior for $\bm{\beta},\bm{\gamma},\sigma,\bm{v}$ given $\bm{y}$, is 
\vspace{-4mm}
\begin{align*}
f(\bm{\beta},\bm{\gamma},\sigma,\bm{v}|\bm{y}) \enskip \propto &  \hspace{1mm}  l(\bm{y}|\bm{\beta},\bm{\gamma},\sigma,\bm{v}) \hspace{1mm} \pi(\bm{\beta}) \hspace{1mm} \pi(\bm{\gamma}) \hspace{1mm} \pi(\sigma), \\
\propto &  \hspace{1mm} \left(\frac{1}{\sigma} \right)^{\frac{3n}{2}} \left( \prod_{i=1}^{n} v_i\right)^{-\frac{1}{2}} \exp\left\{ -\frac{1}{4\sigma} \left[(\bm{y}-\bm{L\beta}-\xi \bm{v})^T \bm{V} (\bm{y}-\bm{L\beta}-\xi \bm{v}) \right] - \frac{\tau(1-\tau)}{\sigma} \sum_{i=1}^n v_i \right\} \\
& \hspace{1mm} \times \hspace{1mm} \exp\left\{-\frac{1}{2\sigma_0^2} (\bm{\beta} - \bm{\beta_0})^T (\bm{\beta} - \bm{\beta_0}) \right\} \\
& \hspace{1mm} \times \hspace{1mm} \exp\left\{-\frac{1}{2\sigma_1^2} \sum_{j=1}^k (\bm{\gamma_j} - \bm{\gamma_{j0}})^T (\bm{\gamma_j} - \bm{\gamma_{j0}}) \right\} \\
& \hspace{1mm} \times \hspace{1mm} \left(\frac{1}{\sigma}\right)^{\frac{a}{2}+1} \exp\left(-\frac{b}{2\sigma} \right).
\end{align*} 
where, $\bm{V}=diag(v_1^{-1},v_2^{-1},\dots,v_n^{-1})$. A Gibbs sampling algorithm is used to generate samples from the analytically intractable posterior distribution $f(\bm{\beta},\bm{\gamma}|\bm{y})$. Some of the full conditionals required in this procedure are available only up to unknown normalizing constants, and we used random walk Metropolis-Hastings algorithm to sample from these full conditional distributions. 

These full conditional distributions are as follows:
\begin{enumerate}[label=(\alph*),leftmargin=*]
\item 
    $\begin{alignedat}[t]{2}
        \pi(\bm{\beta}|\bm{\gamma},\sigma,\bm{v},\bm{y}) \hspace{2mm} \sim & \hspace{1mm} \N \left[ \left(\frac{\bm{L}^T\bm{VL}}{2\sigma} + \frac{\bm{I}}{\sigma_0^2} \right)^{-1} \left(\frac{\bm{L}^T\bm{V}(\bm{y}-\xi\bm{v})}{2\sigma} + \frac{\bm{\beta_0}}{\sigma_0^2}\right), \left(\frac{\bm{L}^T\bm{VL}}{2\sigma} + \frac{\bm{I}}{\sigma_0^2} \right)^{-1} \right]
    \end{alignedat}$
\item 
    $\begin{alignedat}[t]{2}
        \pi(\bm{\gamma_j}|\bm{\beta},\sigma,\bm{v},\bm{y}) \hspace{1mm} \propto & \hspace{1mm}
        \exp\left\{ -\frac{1}{4\sigma} \left[(\bm{y}-\bm{L\beta}-\xi \bm{v})^T \bm{V} (\bm{y}-\bm{L\beta}-\xi \bm{v}) \right] \right\} \hspace{1mm} \times \hspace{1mm}  \exp\left\{-\frac{1}{2\sigma_1^2} (\bm{\gamma_j} - \bm{\gamma_{j0}})^T (\bm{\gamma_j} - \bm{\gamma_{j0}}) \right\} 
    \end{alignedat}$    
\item 
    $\begin{alignedat}[t]{2}
        \pi(\sigma|\bm{\gamma},\bm{\beta},\bm{v},\bm{y}) \hspace{2mm} \sim & \hspace{1mm} \mathbb{IG} \left(\frac{3n+a}{2}, \frac{1}{4} \left[(\bm{y}-\bm{L\beta}-\xi \bm{v})^T \bm{V} (\bm{y}-\bm{L\beta}-\xi \bm{v}) \right] + \tau(1-\tau) \sum_{i=1}^n v_i + \frac{b}{2}\right)
    \end{alignedat}$    
\item 
    $\begin{alignedat}[t]{2}
        \pi(v_i|\bm{\gamma},\bm{\beta},\sigma,\bm{y}) \hspace{1mm} \sim & \hspace{1mm} \mathbb{GIG} \left(\nu,\rho_1,\rho_2 \right) \text{ where, } \nu = \frac{1}{2}, \rho_1^2 = \frac{1}{2\sigma} \left(y_i - \bm{L_i\beta} \right)^2, \text{ and } \rho_2^2 = \frac{1}{2\sigma}.
    \end{alignedat}$ \vspace{2mm}
    \\
The generalized inverse Gaussian distribution is defined as, if $x \sim \mathbb{GIG}\left(\nu,\rho_1,\rho_2 \right)$ then the probability density function of $x$ is given by
    \begin{equation*}
    f(x|\nu,\rho_1,\rho_2) = \frac{(\rho_2/\rho_1)^\nu}{2K_\nu(\rho_1\rho_2)} x^{\nu-1} \exp \left\{-\frac{1}{2}(x^{-1}\rho_1^2 + x\rho_2^2) \right\},
    \end{equation*}
    where $x>0, -\infty < \nu < \infty, \rho_1,\rho_2 \geq 0$ and $K_\nu(.)$ is a modified Bessel function of the third kind (see, \citet{Barndorff-Nielsen-Shephard-2001}).
\end{enumerate}

Unlike the parsimonious parametric models, the Bayesian nonparametric models require additional statistical justification for their theoretical validity. For that reason we are going to provide asymptotic consistency of the posterior distribution derived in our proposed neural network model.

\subsection{Theoretical Validity: Posterior Consistency} 
Let $(\bm{x_1},y_1),\dots,(\bm{x_n},y_n)$ be the given data, let $f_0(\bm{x})$ be the underlying density of $\bm{X}$. Let $Q_{\tau}(y |\bm{X} = \bm{x}) = \mu_0(\bm{x})$ be the true conditional quantile function of $Y$ given $\bm{X}$, and let $\hat{\mu}_n(\bm{x})$ be the estimated conditional quantile function.
\begin{definition}
$\hat{\mu}_n(\bm{x})$ is \textit{asymptotically consistent} for $\mu_0(\bm{x})$ if 
    \begin{equation*} 
        \int \abs{\hat{\mu}_n(\bm{x})-\mu_0(\bm{x})} \hspace{1mm} f_0(\bm{x})\hspace{1mm} \d \bm{x} \overset{p}{\to} 0   
    \end{equation*}
\end{definition}
We are essentially making use of Markov's inequality to ultimately show that $\hat{\mu}_n(\bm{X}) \overset{p}{\to} \mu_0(\bm{X})$. In similar frequentist sense, \citet{Funahashi-1989} and \citet{Hornik-et-al-1989} have shown the asymptotic consistency of the neural networks for mean-regression models by showing the existence of some neural network, $\hat{\mu}_n(\bm{x})$, whose mean squared error with the true function, $\mu_0(\bm{x})$, converges to $0$ in probability.

We will consider the notion of posterior consistency for Bayesian non-parametric problems which is quantified by concentration around the true density function (see \citet{Wasserman-1998}, \citet{Barron-et-al-1999}). This boils down to the above definition of consistency on the conditional quantile functions. The main idea is that the density functions deal with the joint distribution of $\bm{X}$ and $Y$, while the conditional quantile function deals with the conditional distribution of $Y$ given $\bm{X}$. This conditional distribution can then be used to construct the joint distribution by assuming  certain regularity condition on the distribution of $\bm{X}$. This allows the use of some techniques developed in density estimation field. Some of the ideas presented here can be found in \citet{Lee-2000} which developed the consistency results for non-parametric regression using single hidden-layer feed forward neural networks. 

 Let the posterior distribution of the parameters be denoted by $P(.|(\bm{X_1},Y_1),\dots,(\bm{X_n},Y_n))$.  Let $f(\bm{x},y)$ and $f_0(\bm{x},y)$ denote the joint density function of $\bm{x}$ and $y$ under the model and the truth respectively. Indeed, one can construct the joint density $f(\bm{x},y)$ from the condition quantile function $f(y|\bm{x})$ by taking $f(\bm{x},y)=f(y|\bm{x})f(\bm{x})$ where $f(\bm{x})$ denotes the underlying density of $X$. Since, one is only interested in $f(y|\bm{x})$ and $\bm{X}$ is ancillary to the estimation of $f(y|\bm{x})$, one can  use some convenient distribution for $f(x)$. Similar to \citet{Lee-2000}, we define Hellinger neighborhoods of the true density function $f_0(\bm{x},y)=f_0(y|\bm{x})f_0(\bm{x})$ which allows us to quantify the consistency of the posterior. The Hellinger distance between $f_0$ and any joint density function $f$ of $x$ and $y$ is defined as follows.
\begin{equation}\label{Hellinger_distance}
    D_H(f,f_0) = \sqrt{\int \int \left( \sqrt{f(\bm{x},y)} - \sqrt{f_0(\bm{x},y)} \right)^2 \hspace{1mm} \d \bm{x} \hspace{1mm} dy}
\end{equation}
Based on \autoref{Hellinger_distance}, an $\epsilon$-sized Hellinger neighborhood of the true density function $f_0$ is given by
\begin{equation}\label{Hellinger neighborhood}
    A_\epsilon = \{f:\hspace{1mm} D_H(f,f_0)\leq \epsilon \}
\end{equation}

\begin{definition}[Posterior Consistency]
Suppose $(\bm{X_i},Y_i) \sim f_0$. The posterior is \textit{asymptotically consistent for $f_0$ over Hellinger neighborhoods} if $\forall \epsilon > 0$,
    \begin{equation*} 
        P(A_\epsilon|(\bm{X_1},Y_1),\dots,(\bm{X_n},Y_n)) \overset{p}{\to} 1
    \end{equation*}
\textit{i.e. the posterior probability of any Hellinger neighborhood of $f_0$ converges to 1 in probability.}
\end{definition}

Similar to \citet{Lee-2000}, we will prove the asymptotic consistency of the posterior for neural networks with number of hidden nodes, $k$, being function of sample size, $n$. This sequence of models indexed with increasing sample size is called \textit{sieve}. We take sequence of priors, $\{\pi_n\}$, where each $\pi_n$ is defined for a neural network with $k_n$ hidden nodes in it. The predictive density (Bayes estimate of $f$) will then be given by
\begin{equation}\label{predictive density}
    \hat{f}_n(.) = \int f(.)\hspace{1mm}\d P(f|(\bm{X_1},Y_1),\dots,(\bm{X_n},Y_n))
\end{equation}
Let $\mu_0(\bm{x}) = Q_{\tau,f_0}(Y |\bm{X} = \bm{x})$ be the true conditional quantile function and let $\hat{\mu}_n(\bm{x}) = Q_{\tau,\hat{f}_n}(Y |\bm{X} = \bm{x})$ be the posterior predictive conditional quantile function using a neural network. For notational convenience we are going to drop $\bm{x}$ and denote these functions as $\mu_0$ and $\hat{\mu}_n$ occasionally. 


The following is the key result in this case.
\begin{theorem}\label{thm-1}
Let the prior for the regression parameters, $\pi_n$, be an independent normal with mean $0$ and variance $\sigma_0^2$ (fixed) for each of the parameters in the neural network. Suppose that the true conditional quantile function is either continuous or square integrable. Let $k_n$ be the number of hidden nodes in the neural network, and let
$k_n\to \infty$. If there exists a constant $a$ such that $0 < a < 1$
and $k_n \leq n^a$, then $\int |\hat{\mu}_n(\bm{x})-\mu_0(\bm{x})| \hspace{1mm} \d \bm{x} \overset{p}{\to} 0$ as $n \to \infty$
\end{theorem}

In order to prove \autoref{thm-1}, we assume that  $\bm{X_i}\sim U(0,1)$, i.e. density function of $\bm{x}$ is identically equal to 1. This implies joint densities $f(\bm{x},y)$ and $f_0(\bm{x},y)$ are equal to the conditional density functions, $f(y|\bm{x})$ and $f_0(y|\bm{x}
)$ respectively.  Next, we define Kullback-Leibler distance to the true density $f_0(\bm{x},y)$ as follows
\begin{equation}\label{KL Distance}
    D_K(f_0,f) = \mathbb{E}_{f_0} \left[\log \frac{f_0(\bm{X},Y)}{f(\bm{X},Y)} \right]
\end{equation}
Based on \autoref{KL Distance}, a $\delta-$ sized neighborhood  of the true density $f_0$ is given by
\begin{equation} \label{KL neighborhood}
    K_\delta = \{f: D_K(f_0,f) \leq \delta \} 
\end{equation}
 Further towards the proof of  \autoref{thm-1}, we define the sieve $\F_n$ as the set of all neural networks with each parameter less than $C_n$ in absolute value,
\begin{equation} \label{sieve}
    |\gamma_{jh}| \leq C_n,\enskip |\beta_j| \leq C_n, \quad j=0,\dots,k_n, \enskip h=0,\dots,p
\end{equation}
where $C_n$ grows with $n$ such that $C_n \leq \exp (n^{b-a})$ for any constant $b$ where $0<a<b<1$, and $a$ is same as in \autoref{thm-1}.  

For the above choice of sieve, we next provide a set of conditions on the prior $\pi_n$ which guarantee the posterior consistency of $f_0$ over the Hellinger neighborhoods. At the end of this section, we demonstrate that the following theorem and corollary serve as an important tool towards the proof of \autoref{thm-1}.

\begin{theorem}\label{thm-2}
Suppose a prior $\pi_n$ satisfies
    \begin{enumerate}[label=\roman*]
        \item $\exists \hspace{1mm} r > 0 $ and $N_1$ s.t. $\pi_n(\F_n^c)<\exp (-nr), \enskip \forall n\geq N_1$
        \item $\forall \delta,\nu \hspace{1mm} > 0, \exists \hspace{1mm} N_2$ s.t. $\pi_n(K_\delta) \geq \exp (-n\nu),\hspace{1mm} \forall n \geq N_2$.  
    \end{enumerate}
Then $\forall \epsilon>0$,
    \begin{equation*} 
        P(A_\epsilon |(\bm{X_1},Y_1),\dots,(\bm{X_n},Y_n)) \overset{p}{\to} 1
    \end{equation*}
    where $A_\epsilon$ is the Hellinger neighborhood of $f_0$ as in \autoref{Hellinger neighborhood}.
\end{theorem}

\begin{corollary}\label{corollary}
Under the conditions of \autoref{thm-2}, $\hat{\mu}_n$ is asymptotically consistent for $\mu_0$, i.e.
    \begin{equation*} 
        \int \abs{\hat{\mu}_n(x)-\mu_0(x)} \hspace{1mm} \d \bm{x} \overset{p}{\to} 0
    \end{equation*}
\end{corollary}
We present the proofs of Theorem \ref{thm-2} and Corollary \ref{corollary} in \ref{thm2_proof} and \ref{corollary_proof}. In the next few paragraphs, we provide an outline of the steps used.

The main idea behind the proof of \autoref{thm-2} is to consider the complement of  $P(A_\epsilon |(\bm{X_1},Y_1),..,(\bm{X_n},Y_n))$ as a ratio of integrals. Hence let
\begin{equation}
\label{rnf}
    R_n(f) = \frac{{\displaystyle \prod\limits_{i=1}^n f(\bm{x_i},y_i)}}{{\displaystyle \prod\limits_{i=1}^n f_0(\bm{x_i},y_i)}}
\end{equation}
Then
\begin{align}\label{compliment thm-2}
    P(A_\epsilon^c |(\bm{X_1},Y_1),\dots,(\bm{X_n},Y_n)) & = \frac{{\displaystyle\int_{A_\epsilon^c} \prod\limits_{i=1}^n f(\bm{x_i},y_i) \d \pi_n(f)}}{{\displaystyle \int \prod\limits_{i=1}^n f(\bm{x_i},y_i) \d \pi_n(f)}}= \frac{{\displaystyle\int_{A_\epsilon^c} R_n(f) \d \pi_n(f)}}{{\displaystyle\int  R_n(f) \d \pi_n(f)}}\nonumber \\
    &=\frac{{\displaystyle\int_{A_\epsilon^c \cap \F_n} R_n(f) \d \pi_n(f)}+{\displaystyle\int_{A_\epsilon^c \cap \F_n^c} R_n(f) \d \pi_n(f)}}{{\displaystyle\int  R_n(f) \d \pi_n(f)}} \nonumber
\end{align}
In the proof, we will show that the numerator is small as compared to the denominator, thereby ensuring $P(A_\epsilon^c |(\bm{X_1},Y_1),\dots,(\bm{X_n},Y_n))\overset{p}{\to} 0$. The convergence of the second term in the numerator uses assumption i) of the Theorem \ref{thm-2}. It systematically shows that $\int_{\F_n^c} R_n(f) \d \pi_n(f) < \exp(-nr/2)$ except on a set with probability tending to zero (see Lemma \ref{lemma-3} in \ref{AppendixA} for further details). The denominator is bounded using assumption ii) of Theorem \ref{thm-2}. First the KL distance between $f_0$ and $\int f d\pi_n(f)$ is bounded and subsequently used to prove that $P(R_n(f)\leq e^{-n\varsigma}) \overset{p}{\to} 0$, where $\varsigma$ depends on $\delta$ defined earlier. This leads to a conclusion that for all $\varsigma > 0$ and for sufficiently large $n$, $\int R_n(f) d\pi_n(f) > e^{-n\varsigma}$ except on a set of probability going to zero. The result in this case has been condensed in Lemma \ref{lemma-4} presented in \ref{AppendixA}.

Lastly, the first term in the numerator is bounded using the Hellinger bracketing entropy defined below
\begin{definition}[Bracketing Entropy]
For any two functions $l$ and $u$, define the bracket $[l,u]$ as the set of all functions $f$ such that $l\leq f\leq u$. Let $\norm{.}$ be a metric. Define an $\epsilon$-bracket as a bracket with $\norm{u-l}<\epsilon$. Define the bracketing number of a set of functions $\F^*$ as the minimum number of $\epsilon$-brackets needed to cover the $\F^*$, and denote it by $N_{[]}(\epsilon,\F^*,\norm{.})$. Finally, the bracketing entropy, denoted by $H_{[]}()$ , is the natural logarithm of the bracketing number. \citep{Pollard-1991} 
\end{definition}

\citet[Theorem~1, pp.348-349]{Wong-Shen-1995} gives the conditions on the rate of growth of the Hellinger bracketing entropy in order to ensure ${\displaystyle\int_{A_\epsilon^c \cap \F_n} R_n(f) \d \pi_n(f)} \overset{p}{\to} 0$. We next outline the steps to bound the bracketing entropy induced by the sieve structure in \autoref{sieve}.

In this direction, we first compute the covering number and use it as an upper bound in order to find the bracketing entropy for a neural network. Let's consider k, number of hidden nodes to be fixed for now and restrict the parameter space to $\F_n$ then $\F_n \subset \R^d$ where $d=(p+2)k+1$. Further let the covering number be $N(\epsilon,\F_n,\norm{.})$ and use $L_\infty$ as a metric to cover the $\F_n$ with balls of radius $\epsilon$. Then, one does not require more than $((C_n+1)/\epsilon)^d$ such balls which implies
\begin{equation}\label{Covering number}
    N(\epsilon,\F_n,L_\infty) \leq \left(\frac{2C_n}{2\epsilon}+1\right)^d = \left( \vphantom{\frac{C_n+1}{\epsilon} } \frac{C_n+\epsilon}{\epsilon} \right)^d \leq \left(\frac{C_n+1}{\epsilon} \right)^d
\end{equation}
Together with \autoref{Covering number}, we use results from \citet[Theorem~2.7.11 on p.164]{van_der_Vaart-Wellner-1996} to bound the bracketing number of $\F^*$ (the space of all functions on $x$ and $y$ with parameter vectors lying in $\F_n$) as follows:
\begin{equation}
\label{cover}
    N_{[]}(\epsilon,\F^*,\norm{.}_2) \leq \left( \frac{dC_n^2}{\epsilon} \right)^d
\end{equation}
This allows us to determine the rate of growth of Hellinger bracketing entropy which is nothing but the log of the quantity in \autoref{cover}. For further details, we refer to Lemmas \autoref{lemma-2} and \autoref{lemma-3} in \ref{AppendixA}.

Going back to the proof of  \autoref{thm-1}, we show that the $\pi_n$ in  \autoref{thm-1} satisfies the conditions of  \autoref{thm-2} for $\F_n$ as in \autoref{sieve}. Then, the result of \autoref{thm-1} follows from the Corollary \ref{corollary} which is derived from \autoref{thm-2}. Further details of the proof of \autoref{thm-1} are presented in \ref{thm1_proof}. Although \autoref{thm-1} uses a fixed prior, the results can be extended to a more general class of prior distributions as long as the assumptions of \autoref{thm-2} hold.

\section{Empirical Studies} 

\label{Section4} 



\subsection{Simulation Studies}
We investigate the performance of the proposed BQRNN method using two simulated examples and compare the estimated conditional quantiles of the response variable against frequentist quantile regression (QR), Bayesian quantile regression (BQR), and quantile regression neural network (QRNN) models. We implement QR from \href{https://cran.r-project.org/web/packages/quantreg/quantreg.pdf}{\texttt{quantreg}} package, BQR from \href{https://cran.r-project.org/web/packages/bayesQR/bayesQR.pdf}{\texttt{bayesQR}} package and QRNN from \href{https://cran.r-project.org/web/packages/qrnn/qrnn.pdf}{\texttt{qrnn}} \citep{Cannon-2011} package available in R. We choose two simulation scenarios, (i) a linear additive model, (ii) a nonlinear polynomial model. In both the cases we consider a heteroscedastic behavior of y given $\bm{x}$.

\textbf{Scenario 1}: Linear heteroscedastic; Data are generated from
\begin{equation*}
    Y = \bm{X}^T\beta_1 + \bm{X}^T\beta_2\varepsilon,
\end{equation*}

\textbf{Scenario 2}: Non-linear heteroscedastic; Data are generated from
\begin{equation*}
    Y = (\bm{X}^T\beta_1)^4 + (\bm{X}^T\beta_2)^2\varepsilon,
\end{equation*}
where, $\bm{X}=(X_1,X_2,X_3)$ and $X_i$'s are independent and follow $U(0,5)$. The parameters $\beta_1$ and $\beta_2$ are set at $(2,4,6)$ and $(0.1,0.3,0.5)$ respectively.

The robustness of our method is illustrated using three different types of random error component $(\varepsilon)$: $\N(0,1)$, $U(0,1)$, and $\mathcal{E}(1)$ where, $\mathcal{E}(\zeta)$ is the exponential distribution with mean $\zeta^{-1}$. Fore each scenario, we generate 200 independent observations.

We work with a single layer feedforward neural network with a fixed number of nodes $k$. We have tried several values of $k$ in the range of 2-8 and settled on $k=4$ which yielded better results than other choices while bearing reasonable computational cost. We generated 100000 MCMC samples and then discarded first half of the sampled chain as burn-in period.  The $50\%$ burn-in samples in MCMC simulations is not quite unusual and has been suggested by \citet{Gelman-Rubin-1992}. We also choose every $10^{\text{th}}$ sampled value for the estimated parameters to diminish the effect of autocorrelation between consecutive draws. Convergence of MCMC was checked using standard MCMC diagnostic tools \citep{Gelman-et-al-2013}. 

We have tried several different values of the hyperparameters. For brevity, we report the results for only choice of hyperparameters given by $\bm{\beta_0} = \bm{0}$, $\sigma_0^2 = 100$, $\bm{\gamma_j} = \bm{0}$, $\sigma_1^2 = 100$, $a=3$, and $b=0.1$. This particular choice of hyperparameters reflect our preference for near-diffuse priors since in many of the real applications of neural network we don't have information about the input and output variables relationship. Therefore, we wanted to test our model performance in the absence of specific prior elicitation.
We also tried different starting values for $\bm{\beta}$ and $\bm{\gamma}$ chains and found that model output is robust to different starting values of $\bm{\beta}$ but it varies noticeably for different starting values of $\bm{\gamma}$. Further, we observed that our model yields optimal results when we use QRNN estimates of $\bm{\gamma}$ as its starting value in our model. We also have to fine-tune the step size of random walk Metropolis-Hastings (MH) updates in the $\bm{\gamma}$ generation process and settled on random walk variance of $0.01^2$ for scenario 1 while $0.001^2$ for scenario 2. These step sizes lead to reasonable rejection rates for MH sampling of $\bm{\gamma}$ values. However, they indicate the slow traversal of the parameter space for $\bm{\gamma}$ values.

To compare the model performance of QR, BQR, QRNN and BQRNN, we have calculated the theoretical conditional quantiles and contrasted them with the estimated conditional quantiles from the given simulated models. For scenarios 1 and 2, \autoref{Sim_study_table_1} and \autoref{Sim_study_table_2}, respectively, present these results at quantile levels, $\tau = (0.05,0.50,0.95)$ for 3 observations. The \autoref{Sim_study_table_1} indicates neural network models performs comparably with the linear models. This ensures the use of neural network models even if the underlying relationship is linear. In \autoref{Sim_study_table_2}, we can observe that BQRNN outperforms other models in the tail area, i.e. $\tau=0.05,0.95$, whereas it's performance is comparable to QRNN at the median. The natural advantage of our Bayesian procedure over the frequentist QRNN is that we have posterior variance for our conditional quantile estimates which can be used as a uncertainty quantification. 

\begin{sidewaystable}[ph!] 
    \begin{center}
    \caption{Simulated Conditional Quantiles for QR, BQR, QRNN and BQRNN for Simulation Study 1}
    \label{Sim_study_table_1}
    \begin{tabular}{cccccccccc}
        \toprule
        Noise & Quantile & Obs & Theo Cond Q & QR Cond Q & BQR Cond Q & SD(BQR) & QRNN Cond Q & BQRNN Cond Q & SD(BQRNN) \\
        \midrule
        $\varepsilon \sim \N(0,1)$ & 
            $\tau$ = 0.05 & 20th & 17.56 & 17.19 & 17.20 & 0.61 & 16.98 & \textbf{16.40} & \textbf{0.35} \\
            &  & 50th & 37.38 & 37.69 & 37.68 & 0.72 & 38.80 & \textbf{40.99} & \textbf{0.67} \\
            &  & 100th & 42.53 & 42.45 & 42.58 & 0.89 & 41.43 & \textbf{42.07} & \textbf{0.64} \\
            & $\tau$ = 0.50 & 20th & 20.23 & 20.23 & 20.24 & 0.37 & 20.32 & \textbf{17.78} & \textbf{0.51} \\
            &  & 50th & 42.62 & 42.87 & 42.65 & 0.31 & 42.56 & \textbf{39.86} & \textbf{0.77} \\
            &  & 100th & 48.78 & 48.81 & 48.65 & 0.40 & 46.27 & \textbf{45.54} & \textbf{0.97} \\
            & $\tau$ = 0.95 & 20th & 22.90 & 21.81 & 22.38 & 0.67 & 21.42 & \textbf{21.66} & \textbf{0.44} \\
            &  & 50th & 47.86 & 47.92 & 47.69 & 0.70 & 47.52 & \textbf{45.82} & \textbf{0.74} \\
            &  & 100th & 55.02 & 54.28 & 54.20 & 0.92 & 53.80 & \textbf{52.18} & \textbf{1.06} \\
        \midrule
        $\varepsilon \sim U(0,1)$ &  
            $\tau$ = 0.05 & 20th & 20.31 & 20.39 & 20.26 & 0.41 & 20.55 & \textbf{20.18} & \textbf{0.30} \\
            &  & 50th & 42.78 & 42.68 & 42.57 & 0.50 & 42.89 & \textbf{43.38} & \textbf{0.61} \\
            &  & 100th & 48.97 & 48.93 & 48.83 & 0.61 & 48.34 & \textbf{49.37} & \textbf{0.70} \\
            & $\tau$ = 0.50 & 20th & 21.04 & 21.28 & 21.25 & 0.19 & 20.98 & \textbf{20.69} & \textbf{0.32} \\
            &  & 50th & 44.21 & 44.30 & 44.22 & 0.22 & 43.95 & \textbf{43.67} & \textbf{0.64} \\
            &  & 100th & 50.68 & 50.88 & 50.79 & 0.26 & 50.39 & \textbf{53.09} & \textbf{0.74} \\
            & $\tau$ = 0.95 & 20th & 21.77 & 21.87 & 22.09 & 0.44 & 21.72 & \textbf{21.65} & \textbf{0.32} \\
            &  & 50th & 45.65 & 45.58 & 45.68 & 0.49 & 45.43 & \textbf{45.39} & \textbf{0.65} \\
            &  & 100th & 52.39 & 52.36 & 52.45 & 0.55 & 52.49 & \textbf{52.86} & \textbf{0.77} \\
        \midrule
        $\varepsilon \sim \mathcal{E}(1)$ & 
            $\tau$ = 0.05 & 20th & 20.31 & 20.24 & 19.96 & 0.43 & 19.94 & \textbf{19.83} & \textbf{0.32} \\
            &  & 50th & 42.78 & 42.92 & 42.93 & 0.55 & 42.98 & \textbf{44.58} & \textbf{0.63} \\
            &  & 100th & 48.97 & 49.06 & 49.05 & 0.64 & 47.24 & \textbf{48.31} & \textbf{0.81} \\
            & $\tau$ = 0.50 & 20th & 21.36 & 20.95 & 21.04 & 0.23 & 21.11 & \textbf{24.77} & \textbf{0.40} \\
            &  & 50th & 44.83 & 45.54 & 45.47 & 0.27 & 46.30 & \textbf{45.49} & \textbf{0.76} \\
            &  & 100th & 51.41 & 51.93 & 51.91 & 0.40 & 51.68 & \textbf{51.13} & \textbf{0.92} \\
            & $\tau$ = 0.95 & 20th & 25.09 & 25.08 & 25.14 & 0.80 & 23.73 & \textbf{22.73} & \textbf{0.47} \\
            &  & 50th & 52.17 & 53.15 & 52.61 & 0.90 & 54.98 & \textbf{48.48} & \textbf{1.09} \\
            &  & 100th & 60.15 & 61.24 & 60.74 & 1.08 & 59.87 & \textbf{64.54} & \textbf{1.36} \\
        \bottomrule
    \end{tabular}
  \end{center}
\end{sidewaystable}

\begin{sidewaystable}[ph!] 
    \begin{center}
    \caption{Simulated Conditional Quantiles for QR, BQR, QRNN and BQRNN for Simulation Study 2}
    \label{Sim_study_table_2}
    \begin{tabular}{cccrrrrrrr}
        \toprule
        Noise & Quantile & Obs & Theo Cond Q & QR Cond Q & BQR Cond Q & SD(BQR) & QRNN Cond Q & BQRNN Cond Q & SD(BQRNN) \\
        \midrule
        $\varepsilon \sim \N(0,1)$ & 
            $\tau$ = 0.05 & 20th & 167491.82 & -195687.34 & 10208.14 & 33.06 & 30400.87 & \textbf{181466.74} & \textbf{7400.86} \\
            &  & 50th & 3298781.46 & 2107072.28 & 24949.81 & 72.93 & 2626269.69 & \textbf{3310280.70} & \textbf{48816.57} \\
            &  & 100th & 5660909.58 & 2741102.73 & 25681.25 & 75.72 & 3312208.28 & \textbf{5376155.72} & \textbf{97349.16} \\
            & $\tau$ = 0.50 & 20th & 167496.16 & 83421.40 & 92938.13 & 102.25 & 183185.47 & \textbf{167739.24} & \textbf{2479.10} \\
            &  & 50th & 3298798.17 & 2661086.40 & 228324.14 & 282.92 & 3292073.45 & \textbf{3289479.84} & \textbf{46538.33} \\
            &  & 100th & 5660933.30 & 3550782.83 & 233865.24 & 258.28 & 5652427.31 & \textbf{5666982.50} & \textbf{80241.91} \\
            & $\tau$ = 0.95 & 20th & 167500.49 & 1184601.83 & 171360.02 & 166.45 & 193187.39 & \textbf{176826.08} & \textbf{2931.58} \\
            &  & 50th & 3298814.88 & 4731674.95 & 423329.64 & 464.87 & 3309795.50 & \textbf{3301781.56} & \textbf{46949.85} \\
            &  & 100th & 5660957.02 & 5575413.13 & 429134.06 & 433.29 & 5772437.35 & \textbf{5666088.01} & \textbf{80939.97} \\
        \midrule
        $\varepsilon \sim U(0,1)$ &  
            $\tau$ = 0.05 & 20th & 167496.29 & -195833.01 & 10208.01 & 32.89 & -27470.71 & \textbf{172408.55} & \textbf{6272.42} \\
            &  & 50th & 3298798.68 & 2107278.20 & 24949.41 & 72.62 & 2500915.26 & \textbf{3333189.64} & \textbf{47562.30} \\
            &  & 100th & 5660934.02 & 2741400.69 & 25681.03 & 75.38 & 3100036.18 & \textbf{5201780.67} & \textbf{104297.84} \\
            & $\tau$ = 0.50 & 20th & 167497.47 & 83435.89 & 92937.96 & 101.23 & 172793.77 & \textbf{168359.45} & \textbf{2583.07} \\
            &  & 50th & 3298803.25 & 2661086.35 & 228323.11 & 279.51 & 3314585.80 & \textbf{3296954.08} & \textbf{46650.01} \\
            &  & 100th & 5660940.51 & 3550796.98 & 233864.74 & 256.79 & 5607084.86 & \textbf{5661907.57} & \textbf{80101.17} \\
            & $\tau$ = 0.95 & 20th & 167498.66 & 1184587.54 & 171359.88 & 166.91 & 196767.97 & \textbf{175212.17} & \textbf{3312.33} \\
            &  & 50th & 3298807.82 & 4731690.36 & 423328.99 & 463.93 & 3316486.10 & \textbf{3303062.26} & \textbf{46776.32} \\
            &  & 100th & 5660947.00 & 5575416.01 & 429133.41 & 430.65 & 5757769.91 & \textbf{5663521.46} & \textbf{80219.88} \\
        \midrule
        $\varepsilon \sim \mathcal{E}(1)$ & 
            $\tau$ = 0.05 & 20th & 167496.29 & -195784.51 & 10208.23 & 33.02 & -172912.92 & \textbf{124688.39} & \textbf{12531.19} \\
            &  & 50th & 3298798.69 & 2107220.10 & 24949.87 & 73.13 & 2380587.55 & \textbf{3284717.78} & \textbf{49843.67} \\
            &  & 100th & 5660934.04 & 2741316.61 & 25681.36 & 75.71 & 2825010.54 & \textbf{5177969.08} & \textbf{97539.50} \\
            & $\tau$ = 0.50 & 20th & 167497.98 & 83442.50 & 92938.27 & 102.50 & 175113.10 & \textbf{170058.16} & \textbf{2557.19} \\
            &  & 50th & 3298805.21 & 2661099.63 & 228323.86 & 283.61 & 3309978.81 & \textbf{3296869.79} & \textbf{46636.29} \\
            &  & 100th & 5660943.29 & 3550827.77 & 233865.06 & 259.37 & 5603733.56 & \textbf{5662588.82} & \textbf{80087.88} \\
            & $\tau$ = 0.95 & 20th & 167504.05 & 1184595.04 & 171360.22 & 164.15 & 182663.10 & \textbf{175802.87} & \textbf{2880.89} \\
            &  & 50th & 3298828.60 & 4731683.78 & 423330.52 & 458.74 & 3337299.83 & \textbf{3306641.01} & \textbf{46833.26} \\
            &  & 100th & 5660976.50 & 5575416.49 & 429135.00 & 423.28 & 5838226.18 & \textbf{5686965.11} & \textbf{80976.40} \\
        \bottomrule
    \end{tabular}
  \end{center}
\end{sidewaystable}

\subsection{Real Data Examples}
In this section, we apply our proposed method to three real world datasets which are publicly available. These datasets will be used to demonstrate the performance of nonlinear regression methods.

The first dataset is widely used Boston Housing dataset which is available in R package MASS \citep{Venables-Ripley-2002}. It contains 506 census tracts of Boston Standard Metropolitan Statistical Area in 1970. There are 13 predictor variables and one response variable, corrected median value of owner-occupied homes (in USD 1000s). Predictor variables include per capita crime rate by town, proportion of residential land zoned for lots over 25,000 sq.ft., nitrogen oxide concentration, proportion of owner-occupied units built prior to 1940, full-value property-tax rate per $\$$10,000, and lower status of the population in percent, among others. There is high correlation among some of these predictor variables and the goal here is to determine the best fitting functional form which would help in improving the housing value forecasts.

The second dataset is the Gilgais dataset available in R package MASS. This data was collected on a line transect survey in gilgai territory in New South Wales, Australia. Gilgais are repeated mounds and depressions formed on flat land, and many-a-times are regularly distributed. The data collection with 365 sampling locations on a linear grid of 4 meters spacing aims to check if the gilgai patterns are reflected in the soil properties as well. At each of the sampling location, samples were taken at depths 0-10 cm, 30-40 cm and 80-90 cm below the surface. The input variables included pH, electrical conductivity and chloride content and were measured on a 1:5 soil:water extract from each sample. Here, the response variable is e80 (electrical conductivity in mS/cm: 80–90 cm) and we will focus on finding the true functional relationship present in the dataset.

The third dataset is concrete data which is compiled by \citet{Yeh-1998} and is available on UCI machine learning repository. It consists of 1030 records, each containing 8 input features and compressive strength of concrete as an output variable. The input features include the amounts of  ingredients in  high performance concrete (HPC) mixture which are cement, fly ash, blast furnace slag, water, superplasticizer, coarse aggregate, and fine aggregate. Moreover, age of the mixture in days is also included as one of the predictor variable. According to \citet{Yeh-1998}, the compressive strength of concrete is a highly non-linear function of the given inputs. The central purpose of the study is to predict the compressive strength of HPC using the input variables. 

In our experiments, we compare the performance of QR, BQR, QRNN, and BQRNN estimates for $f(\bm{x})$, the true functional form of the data, in both training and testing data using mean check function (or, mean tilted absolute loss function). The mean check function (MCF) is given as
\begin{equation*}
    \text{MCF} = \frac{1}{N} \sum_{i=1}^N  \rho_\tau (y_i-\hat{f}(\bm{x}_i))
\end{equation*}
where, $\rho_\tau(.)$ is defined in \autoref{check_fun} and $\hat{f}(\bm{x})$ is an estimate of $f(\bm{x})$. We resort to this comparison criterion since we don't have the theoretical conditional quantiles for the data at our disposal. For each dataset, we randomly choose $80\%$ of data points for training the model and then remaining $20\%$ is used to test the prediction ability of the fitted model. Our single hidden-layer neural network has $k=4$ hidden layer nodes and the random walk variance is chosen to be $0.01^2$. These particular choices of the number of hidden layer nodes and random walk step size are based on their optimal performance among several different choices while providing reasonable computational complexity. We perform these analyses for quantiles, $\tau = (0.05,0.25,0.50,0.75,0.95)$, and present the model comparison results for both training and testing data in \autoref{Real_data_table}.

It can be seen that our model performs comparably well with QRNN model while outperforming linear models (QR and BQR) in all the datasets. We can see that both QRNN and BQRNN have lower mean check function values for training data than their testing counterpart. This suggests that neural networks may be overfitting the data while trying to find the true underlying functional form.  The model performance of QR and BQR models is inferior compared to neural network models, particularly when the regression relationship is non-linear. Furthermore,  the Bayesian quantile regression neural network model provides uncertainty estimation as a natural byproduct which is not available in the frequentist QRNN model.

\begin{table}[ph!] 
  \begin{center}
    \caption{Real Data Applications Comparison Results (MCF Values)}
    \label{Real_data_table}
    \begin{tabular}{cccrrrr}
        \toprule
        Noise & Quantile & Sample & QR & BQR & QRNN & BNNQR \\ 
        \midrule
        Boston & 
            $\tau$ = 0.05 & Train & 0.3009 & 0.3102 & 0.2084 & \textbf{0.1832} \\
            &  & Test & 0.3733 & 0.3428 & 0.3356 & \textbf{0.5842} \\
            & $\tau$ = 0.25 & Train & 1.0403 & 1.0431 & 0.6340 & \textbf{0.6521} \\
            &  & Test & 1.2639 & 1.2431 & 1.0205 & \textbf{1.2780} \\
            & $\tau$ = 0.50 & Train & 1.4682 & 1.4711 & 0.8444 & \textbf{0.8864} \\
            &  & Test & 1.8804 & 1.8680 & 1.4638 & \textbf{1.5882} \\
            & $\tau$ = 0.75 & Train & 1.3856 & 1.3919 & 0.6814 & \textbf{0.7562} \\
            &  & Test & 1.8426 & 1.8053 & 1.3773 & \textbf{1.4452} \\
            & $\tau$ = 0.95 & Train & 0.5758 & 0.6009 & 0.2276 & \textbf{0.2206} \\
            &  & Test & 0.7882 & 0.6174 & 0.8093 & \textbf{0.6880} \\
        \midrule
        Gilgais & 
            $\tau$ = 0.05 & Train & 3.6610 & 3.7156 & 3.0613 & \textbf{2.7001} \\
            &  & Test & 3.4976 & 3.2105 & 2.9137 & \textbf{3.9163} \\
            & $\tau$ = 0.25 & Train & 13.9794 & 14.5734 & 8.6406 & \textbf{8.4565} \\
            &  & Test & 11.8406 & 11.4832 & 9.3298 & \textbf{10.1386} \\
            & $\tau$ = 0.50 & Train & 18.1627 & 21.3587 & 10.4667 & \textbf{10.7845} \\
            &  & Test & 15.8210 & 17.2037 & 13.5297 & \textbf{14.3699} \\
            & $\tau$ = 0.75 & Train & 13.6598 & 18.8357 & 7.9679 & \textbf{7.9905} \\
            &  & Test & 12.3926 & 18.2477 & 9.4711 & \textbf{10.6414} \\
            & $\tau$ = 0.95 & Train & 3.8703 & 6.4137 & 2.3289 & \textbf{2.2508} \\
            &  & Test & 4.1266 & 6.4300 & 3.0280 & \textbf{2.5586} \\
        \midrule
        Concrete & $\tau$ = 0.05 & Train & 2.9130 & 4.4500 & 2.0874 & \textbf{2.0514} \\
            &  & Test & 2.9891 & 4.2076 & 2.2021 & \textbf{2.6793} \\
            & $\tau$ = 0.25 & Train & 10.0127 & 14.7174 & 7.0063 & \textbf{7.0537} \\
            &  & Test & 9.6451 & 14.3567 & 6.9179 & \textbf{7.4069} \\
            & $\tau$ = 0.50 & Train & 13.1031 & 19.8559 & 9.3638 & \textbf{9.3728} \\
            &  & Test & 12.7387 & 18.2309 & 9.8936 & \textbf{10.9172} \\
            & $\tau$ = 0.75 & Train & 11.5179 & 17.7680 & 7.6789 & \textbf{7.3932} \\
            &  & Test & 10.8299 & 16.3257 & 8.5755 & \textbf{9.4147} \\
            & $\tau$ = 0.95 & Train & 3.9493 & 6.8747 & 2.4403 & \textbf{2.5262} \\
            &  & Test & 3.6489 & 6.8435 & 2.7768 & \textbf{4.1369} \\
        \bottomrule
    \end{tabular}
  \end{center}
\end{table}

\section{Conclusion} 

\label{Section5} 



The manuscript has developed the Bayesian neural network models for quantile estimation in a systematic way. The practical implementation of Gibbs sampling coupled with Metropolis-Hastings updates method have been discussed in detail. The method exploits the location-scale mixture representation of the asymmetric Laplace distribution which makes its implementation easier. The model can be thought as a hierarchical Bayesian model which makes use of independent normal priors for the neural network weight parameters. We also carried out similar study using a data dependent Zellner's prior \citep{Zellner-1986} for the weights of neural network, but results are not provided here to keep the discussion simple and further theoretical justifications are needed for this prior. A future work in this area could be sparsity induced priors to allow for node and layer selection in multi-layer neural network architecture. 

Further, we have developed asymptotic consistency of the posterior distribution of the neural network parameters. The presented result can be extended to a more general class of prior distributions if they satisfy the \autoref{thm-2} assumptions. Following the theory developed here, we bridge the gap between asymptotic justifications separately available for Bayesian quantile regression and Bayesian neural network regression. The theoretical arguments developed here justify using neural networks for quantile estimation in nonparametric regression problems using Bayesian methods. 

The proposed MCMC procedure has been shown to work when the number of parameters are relatively low compared to the number of observations. 
We noticed that convergence of the posterior chains take long time and there is noticeable autocorrelation left in the sampled chains even after burn-in period. We also acknowledge that our random-walk Metropolis-Hastings algorithm has small step size which might lead to slow traversal of the parameter space ultimately raising the computational cost of our algorithm. The computational complexity in machine learning methods are well-known. Further research is required in these aspects of model implementation.








\appendix


\clearpage

\section{Lemmas for Posterior Consistency Proof} 

\label{AppendixA} 

For all the proofs in \ref{AppendixA} and \ref{AppendixB}, we assume $\bm{X}_{p \times 1}$ to be uniformly distributed on $[0,1]^p$ and keep them fixed. Thus, $f_0(\bm{x})=f(\bm{x})=1$. Conditional on $\bm{X}$, the univariate response variable $Y$ has asymmetric Laplace distribution with location parameter determined by the neural network.  We are going to fix its scale parameter, $\sigma$, to be 1 for the posterior consistency derivations. Thus,
\begin{equation}\label{ALD neural network}
    Y |\bm{X} = \bm{x} \sim ALD \left( \beta_0 + \sum_{j=1}^k \beta_j \frac{1}{1 + \exp{ \left( -\gamma_{j0}- \sum_{h=1}^p \gamma_{jh} x_{h} \right)}}, 1, \tau \right) 
\end{equation}
The number of input variables, $p$, is taken to be fixed while the number of hidden nodes, $k$, will be allowed to grow with the sample size, n.

All the lemmas described below are taken from \citet{Lee-2000}.
\begin{lemma}\label{lemma-1}
Suppose $H_{[]}(u) \leq \log [(C_n^2d_n/u)^{d_n}], d_n=(p+2)k_n+1, k_n \leq n^a$ and $C_n\leq \exp (n^{b-a})$ for $0<a<b<1$. Then for any fixed constants $c,\epsilon >0,$ and for all sufficiently large $n, \int_0^\epsilon \sqrt{H_{[]}(u)} \leq c\sqrt{n}\epsilon^2$.
\end{lemma}
\begin{proof}
The proof follows from the proof of Lemma 1 from \citep[p. 634-635]{Lee-2000} in BQRNN case.
\end{proof}
For Lemmas \autoref{lemma-2}, \autoref{lemma-3} and \autoref{lemma-4}, we make use of the following notations. From \autoref{rnf}, recall \begin{equation*}
    R_n(f) = \prod_{i=1}^n \frac{f(x_i,y_i)}{f_0(x_i,y_i)}
\end{equation*}
is the ratio of likelihoods under neural network density $f$ and the true density $f_0$. $\F_n$ is the sieve as defined in \autoref{sieve} and $A_\epsilon$ is the Hellinger neighborhood of the true density $f_0$ as in \autoref{Hellinger neighborhood}.

\begin{lemma}\label{lemma-2}
$\underset{f\in A_\epsilon^c \cap \F_n}{\sup} R_n(f) \leq 4 \exp(-c_2n\epsilon^2)$ a.s. for sufficiently large $n$.
\end{lemma}
\begin{proof}
Using the outline of the proof of Lemma 2 from \citet[p. 635]{Lee-2000}, first we have to bound the Hellinger bracketing entropy using \citet[Theorem~2.7.11 on p.164]{van_der_Vaart-Wellner-1996}. Next we use Lemma \ref{lemma-1} to show that the conditions of \citet[Theorem~1 on p.348-349]{Wong-Shen-1995} hold and finally we apply that theorem to get the result presented in the Lemma 2. 

In our case of BQRNN, we only need to derive first step using ALD density mentioned in \autoref{ALD neural network}. And rest of the steps follow from the proof given in \citet{Lee-2000}. As we are looking for the Hellinger bracketing entropy for neural networks, we use $L_2$ norm on the square root of the density functions, $f$. The $L_\infty$ covering number was computed above in \autoref{Covering number}, so here $d^*=L_\infty$. The version of \citet[Theorem~2.7.11]{van_der_Vaart-Wellner-1996} that we are interested in is 
\begin{align*}
    & \text{If}  \enskip \abs{\sqrt{f_t(x,y)}-\sqrt{f_s(x,y)}} \leq d^*(s,t) F(x,y) \quad \text{for some } F, \\
    & \text{then,} \enskip N_{[]}(2\epsilon \norm{F}_2,\F^*,\norm{.}_2) \leq N(\epsilon,\F_n,d^*)
\end{align*}
Now let's start by defining some notations,
\begin{align}
     f_t(x,y) & = \tau(1-\tau) \exp \left(-(y-\mu_t(x)) (\tau-I_{(y\leq\mu_t(x))}) \right), \nonumber \\
    & \hspace{1mm} \text{ where,} \enskip \mu_t(x) = \beta_0^t + \sum_{j=1}^k  \frac{\beta_j^t}{1 + \exp{(-A_j(x))}} \enskip \text{and} \enskip A_j(x) = \gamma_{j0}^t + \sum_{h=1}^p \gamma_{jh}^t x_{h} \label{mu_t}\\
    f_s(x,y) & = \tau(1-\tau) \exp \left(-(y-\mu_s(x)) (\tau-I_{(y\leq\mu_s(x))}) \right), \nonumber \\
    & \hspace{1mm} \text{where,} \enskip \mu_s(x) = \beta_0^s + \sum_{j=1}^k  \frac{\beta_j^s}{1 + \exp{(-B_j(x))}} \enskip \text{and} \enskip B_j(x) = \gamma_{j0}^s + \sum_{h=1}^p \gamma_{jh}^s x_{h} \label{mu_s}
\end{align}
For notational convenience, we drop $x$ and $y$ from $f_s(x,y)$, $f_t(x,y)$, $\mu_s(x)$, $\mu_t(x)$, $B_j(x)$, and $A_j(x)$ and denote them as $f_s$, $f_t$, $\mu_s$, $\mu_t$, $B_j$, and $A_j$. 

\begin{align}\label{mod_term}
    \abs{\sqrt{f_t}-\sqrt{f_s}} & = \sqrt{\tau(1-\tau)} \abs{\exp \left(-\frac{1}{2}(y-\mu_t) (\tau-I_{(y\leq\mu_t)}) \right) - \exp \left(-\frac{1}{2}(y-\mu_s) (\tau-I_{(y\leq\mu_s)}) \right)}\nonumber \\
    & \hspace{1mm} \text{As,} \enskip \tau \in (0,1) \text{ is fixed.}\nonumber\\
    & \leq \frac{1}{2} \abs{\exp \left(-\frac{1}{2}(y-\mu_t) (\tau-I_{(y\leq\mu_t)}) \right) - \exp \left(-\frac{1}{2}(y-\mu_s) (\tau-I_{(y\leq\mu_s)}) \right)}
\end{align}
Now let's separate above term into two cases when: (a) $\mu_s \leq \mu_t$ and (b) $\mu_s > \mu_t$. Further let's consider case-$a$ and break it into three subcases when: (i) $y \leq \mu_s \leq \mu_t$, (ii) $\mu_s < y \leq \mu_t$, and (iii) $\mu_s \leq \mu_t < y$.
\vspace{-2mm}
\begin{enumerate}[label=Case-a (\roman*),wide=0pt,leftmargin=.1in] 
    \item $y \leq \mu_s \leq \mu_t$ \\
    \hspace{3mm} The \autoref{mod_term} simplifies to
        \begin{align}\label{calculus ineq}
            \hphantom{\abs{\sqrt{f_t}-\sqrt{f_s}}}& \frac{1}{2} \abs{\exp \left(-\frac{1}{2}(y-\mu_t) (\tau-1) \right) -  \exp \left(-\frac{1}{2}(y-\mu_s) (\tau-1) \right)} \nonumber \\
            & = \frac{1}{2} \abs{\exp \left(-\frac{1}{2}(y-\mu_s) (\tau-1) \right)} \abs{\exp \left(-\frac{1}{2}(\mu_s-\mu_t) (\tau-1) \right) -1} \nonumber \\
            & \quad \text{As first term in modulus is } \leq 1 \nonumber \\
            & \leq \frac{1}{2} \abs{1 - \exp \left(-\frac{1}{2}(\mu_t-\mu_s) (1-\tau) \right)} \nonumber \\
            & \quad \text{Note: } 1-\exp(-z) \leq z \enskip \forall z \in \R \implies \abs{1-\exp(-z)} \leq \abs{z} \enskip \forall z \geq 0 \\
            & \leq \frac{1}{4}\abs{\mu_t-\mu_s}(1-\tau) \nonumber \\
            & \leq \frac{1}{4}\abs{\mu_t-\mu_s} \nonumber \\
            & \leq \frac{1}{2}\abs{\mu_t-\mu_s} \nonumber
        \end{align}
    \item $\mu_s < y \leq \mu_t$ \\
    \hspace{3mm} The \autoref{mod_term} simplifies to
        \begin{align*}
            \hphantom{\abs{\sqrt{f_t}-\sqrt{f_s}}}& \frac{1}{2} \abs{\exp \left(-\frac{1}{2}(y-\mu_t) (\tau-1) \right) -  \exp \left(-\frac{1}{2}(y-\mu_s) \tau \right)}\\
            & = \frac{1}{2} \abs{\exp \left(-\frac{1}{2}(y-\mu_s) (\tau-1) \right) -1 +1 - \exp \left(-\frac{1}{2}(y-\mu_s) \tau \right)} \\
            & \leq  \frac{1}{2} \abs{1-\exp \left(-\frac{1}{2}(y-\mu_t) (\tau-1) \right)} + \frac{1}{2} \abs{1 - \exp \left(-\frac{1}{2}(y-\mu_s) \tau \right)}  \\ 
            & \quad \text{Let's use calculus inequality mentioned in \ref{calculus ineq}} \\
            & \leq \frac{1}{4}\abs{(y-\mu_t)(\tau-1)} + \frac{1}{4}\abs{(y-\mu_s)\tau} \\
            & \quad \text{Both terms are positive so we will combine them in one modulus} \\
            & = \frac{1}{4} \abs{(y-\mu_t)(\tau-1) + (y-\mu_t+\mu_t-\mu_s)\tau} \\
            & = \frac{1}{4} \abs{(y-\mu_t)(2\tau-1) + (\mu_t-\mu_s)\tau} \\
            & \leq \frac{1}{4} \left[\abs{(y-\mu_t)}\abs{2\tau-1} + \abs{\mu_t-\mu_s}\tau \right] \\
            & \text{Here, } \abs{y-\mu_t} \leq \abs{\mu_t-\mu_s} \enskip \text{and} \enskip \abs{2\tau-1} \leq 1 \\
            & \leq \frac{1}{2}\abs{\mu_t-\mu_s} 
        \end{align*}
    \item $\mu_s \leq \mu_t < y$ \\
    \hspace{3mm} The \autoref{mod_term} simplifies to
        \begin{align*}
            \hphantom{\abs{\sqrt{f_t}-\sqrt{f_s}}}& \frac{1}{2} \abs{\exp \left(-\frac{1}{2}(y-\mu_t) \tau \right) -  \exp \left(-\frac{1}{2}(y-\mu_s)\tau \right)} \\
            & = \frac{1}{2} \abs{\exp \left(-\frac{1}{2}(y-\mu_t) \tau \right)} \abs{1-\exp \left(-\frac{1}{2}(\mu_t-\mu_s) \tau \right)}  \\
            & \quad \text{As first term in modulus is } \leq 1 \\
            & \leq \frac{1}{2} \abs{1 - \exp \left( -\frac{1}{2} (\mu_t-\mu_s) \tau \right)} \\
            & \quad \text{Using the calculus inequality mentioned in \ref{calculus ineq}} \hphantom{\text{-------------------------}} \\
            & \leq \frac{1}{4}\abs{\mu_t-\mu_s}\tau \nonumber \\
            & \leq \frac{1}{4}\abs{\mu_t-\mu_s} \nonumber \\
            & \leq \frac{1}{2}\abs{\mu_t-\mu_s} \nonumber            
        \end{align*}
\end{enumerate}
We can similarly bound the \autoref{mod_term} in case-(b) where $\mu_s > \mu_t$ by $\abs{\mu_t-\mu_s}/2$. Now,
\begin{align}
    \abs{\sqrt{f_t}-\sqrt{f_s}} & \leq \frac{1}{2}\abs{\mu_t-\mu_s} \nonumber\\
    & \quad  \text{Now, let's substitute $\mu_t$ and $\mu_s$ from \ref{mu_t} and \ref{mu_s}} \nonumber\\
    & = \frac{1}{2} \abs{\beta_0^t + \sum_{j=1}^k  \frac{\beta_j^t}{1 + \exp{(-A_j)}} - \beta_0^s - \sum_{j=1}^k  \frac{\beta_j^s}{1 + \exp{(-B_j)}}} \nonumber\\
    & \leq \frac{1}{2} \left[ \abs{\beta_0^t - \beta_0^s} + \sum_{j=1}^k  \abs{\frac{\beta_j^t}{1 + \exp{(-A_j)}} -\frac{\beta_j^s}{1 + \exp{(-B_j)}}} \right] \nonumber\\
    & = \frac{1}{2} \left[ \abs{\beta_0^t - \beta_0^s} + \sum_{j=1}^k  \abs{\frac{\beta_j^t-\beta_j^s+\beta_j^s}{1 + \exp{(-A_j)}} -\frac{\beta_j^s}{1 + \exp{(-B_j)}}} \right] \nonumber\\
    & = \frac{1}{2} \left[ \abs{\beta_0^t - \beta_0^s} + \sum_{j=1}^k  \frac{\abs{\beta_j^t-\beta_j^s}}{1 + \exp{(-A_j)}} + \sum_{j=1}^k \abs{\beta_j^s} \abs{\frac{1}{1 + \exp{(-A_j)}} -\frac{1}{1 + \exp{(-B_j)}}} \right] \nonumber\\
    & \quad \text{Recall that} \enskip \abs{\beta_j^s} \leq C_n \nonumber\\
    & \leq \frac{1}{2} \left[ \abs{\beta_0^t - \beta_0^s} + \sum_{j=1}^k  \abs{\beta_j^t-\beta_j^s} + \sum_{j=1}^k C_n \abs{\frac{\exp(-B_j) - \exp(-A_j)}{(1 + \exp (-A_j))(1 + \exp(-B_j))}} \right] \label{exp_diff}
\end{align}
\begin{align*}
    \text{Note:} \enskip \abs{\exp(-B_j) - \exp(-A_j)} & = \left\{ \begin{matrix} \exp(-A_j)(1-\exp(-(B_j-A_j))), & \enskip \text{when} \enskip B_j-A_j \geq 0 \\ exp(-B_j)(1-\exp(-(A_j-B_j))), & \enskip \text{when} \enskip A_j-B_j \geq 0 \end{matrix} \right. \\
    & \quad \text{Using the calculus inequality mentioned in \ref{calculus ineq}} \hphantom{\text{-------------------------}} \\
    & \leq \left\{ \begin{matrix} \exp(-A_j)(B_j-A_j), & \enskip \text{when} \enskip B_j-A_j \geq 0 \\ \exp(-B_j)(A_j-B_j), & \enskip \text{when} \enskip A_j-B_j \geq 0 \end{matrix} \right. 
\end{align*}
\begin{align*}
    \text{So,} \enskip \abs{\frac{\exp(-B_j) - \exp(-A_j)}{(1 + \exp (-A_j))(1 + \exp(-B_j))}} & \leq \left\{ \begin{matrix} \frac{\exp(-A_j)(B_j-A_j)}{(1 + \exp (-A_j))(1 + \exp(-B_j))}, & \enskip \text{when} \enskip B_j-A_j \geq 0 \\ \frac{\exp(-B_j)(A_j-B_j)}{(1 + \exp (-A_j))(1 + \exp(-B_j))}, & \enskip \text{when} \enskip A_j-B_j \geq 0 \end{matrix} \right. \\
    & \leq \abs{A_j-B_j}
\end{align*}
Hence we can bound the \autoref{exp_diff} as follows
\begin{align*}
    \abs{\sqrt{f_t}-\sqrt{f_s}} & \leq \frac{1}{2} \left[ \abs{\beta_0^t - \beta_0^s} + \sum_{j=1}^k  \abs{\beta_j^t-\beta_j^s} + \sum_{j=1}^k C_n \abs{A_j-B_j} \right] \\
    & \quad  \text{Now, let's substitute $A_j$ and $B_j$ from \ref{mu_t} and \ref{mu_s}} \nonumber\\
    & \leq \frac{1}{2} \left[ \abs{\beta_0^t - \beta_0^s} + \sum_{j=1}^k  \abs{\beta_j^t-\beta_j^s} + \sum_{j=1}^k C_n \abs{\gamma_{j0}^t + \sum_{h=1}^p \gamma_{jh}^t x_{h} - \gamma_{j0}^s - \sum_{h=1}^p \gamma_{jh}^s x_{h}} \right] \\
    & \leq \frac{1}{2} \left[ \abs{\beta_0^t - \beta_0^s} + \sum_{j=1}^k  \abs{\beta_j^t-\beta_j^s} + \sum_{j=1}^k C_n \left( \abs{\gamma_{j0}^t - \gamma_{j0}^s} + \sum_{h=1}^p \abs{x_h} \abs{\gamma_{jh}^t - \gamma_{jh}^s } \right) \right] \\
    & \quad \text{Recall that} \enskip \abs{x_h} \leq 1 \enskip \text{and w.l.o.g assume} \enskip C_n > 1 \\
    & \leq \frac{C_n}{2} \left[ \abs{\beta_0^t - \beta_0^s} + \sum_{j=1}^k  \abs{\beta_j^t-\beta_j^s} + \sum_{j=1}^k \left( \abs{\gamma_{j0}^t - \gamma_{j0}^s} + \sum_{h=1}^p \abs{\gamma_{jh}^t - \gamma_{jh}^s } \right) \right] \\
    & \leq \frac{C_n d}{2} \norm{t-s}_\infty
\end{align*}
Now rest of the steps will follow from the proof of Lemma 2 given in \citet[p. 635-636]{Lee-2000}. 
\end{proof}

\begin{lemma}\label{lemma-3}
 If there exists a constant $r>0$ and $N$, such that $\F_n$ satisfies $\pi_n(\F_n^c)<\exp (-nr), \forall n\geq N$, then there exists a constant $c_2$ such that $\int_{A_\epsilon^c} R_n(f) \d \pi_n(f) < \exp(-nr/2) + \exp(-nc_2\epsilon^2)$ except on a set of probability tending to zero.
\end{lemma}
\begin{proof}
The proof is same as the proof of Lemma 3 from \citep[p. 636]{Lee-2000} in BQRNN scenario.
\end{proof}

\begin{lemma}\label{lemma-4}
Let $K_\delta$ be the KL-neighborhood as in \autoref{KL neighborhood}. Suppose that for all  $\delta,\nu \hspace{1mm} > 0, \exists \hspace{1mm} N$ s.t. $\pi_n(K_\delta) \geq \exp (-n\nu),\hspace{1mm} \forall n \geq N$. Then for all $\varsigma > 0$ and sufficiently large $n$, $\int R_n(f) d\pi_n(f) > e^{-n\varsigma}$ except on a set of probability going to zero.
\end{lemma}
\begin{proof}
The proof is same as the proof of Lemma 5 from \citep[p. 637]{Lee-2000} in BQRNN scenario.
\end{proof} 

\begin{lemma}\label{lemma-5}
Suppose that $\mu$ is a neural network regression with parameters $(\theta_1,\dots\theta_d)$, and let $\tilde{\mu}$ be another neural network with parameters $(\tilde{\theta}_1,\dots\tilde{\theta}_{\tilde{d}_n})$. Define $\theta_i=0$ for $i >d$ and $\tilde{\theta}_j=0$ for $j>\tilde{d}_n$. Suppose that the number of nodes of $\mu$ is k, and that the number of nodes of $\tilde{\mu}$ is $\tilde{k}_n=O(n^a)$ for some $a$, $0<a<1$. Let 
\begin{equation}
\label{mdef}
    M_\varsigma =\{\tilde{\mu}\Big| \abs{\theta_i-\tilde{\theta}_i} \leq \varsigma , i=1,2,\dots\}
\end{equation} 
Then for any $\tilde{\mu} \in M_\varsigma$ and for sufficiently large n,
\begin{equation*}
    \underset{x\in \mathcal{X}}{\sup}(\tilde{\mu}(x)-\mu(x))^2 \leq (5n^a)^2\varsigma^2
\end{equation*}
\end{lemma}
\begin{proof}
The proof is same as the proof of Lemma 6 from \citep[p. 638-639]{Lee-2000}.
\end{proof}

\clearpage


\section{Posterior Consistency Theorem Proofs} 

\label{AppendixB} 

\subsection{\autoref{thm-2} Proof}\label{thm2_proof}
For the proof of \autoref{thm-2} and Corollary \ref{corollary}, we use the following notations. From \autoref{rnf}, recall that
\begin{equation*}
    R_n(f) = \prod_{i=1}^n \frac{f(\bm{x}_i,y_i)}{f_0(\bm{x}_i,y_i)}
\end{equation*}
is the ratio of likelihoods under neural network density $f$ and the true  density $f_0$. Also, $\F_n$ is the sieve as defined in \autoref{sieve}. Finally, $A_\epsilon$ is the Hellinger neighborhood of the true density $f_0$ as in \autoref{Hellinger neighborhood}.

By Lemma \autoref{lemma-3}, there exists a constant $c_2$ such that $\int_{A_\epsilon^c} R_n(f) \d \pi_n(f) < \exp(-nr/2) + \exp(-nc_2\epsilon^2)$ for sufficiently large n. Further, from Lemma \autoref{lemma-4}, $\int R_n(f) \d \pi_n(f) \geq \exp(-n\varsigma)$ for sufficiently large n.
\begin{align*}\label{compliment thm-2}
    P(A_\epsilon^c |(\bm{X_1},Y_1),\dots,(\bm{X_n},Y_n)) & = \frac{{\displaystyle\int_{A_\epsilon^c} R_n(f) \d \pi_n(f)}}{{\displaystyle\int  R_n(f) \d \pi_n(f)}}\nonumber \\
    & < \frac{\exp \left( - \frac{nr}{2} \right)+ \exp(-nc_2\epsilon^2)}{\exp(-n\varsigma)} \\
    & = \exp\left(-n \left[\frac{r}{2} -\varsigma \right] \right) + \exp \left( -n\epsilon^2 [c_2 - \varsigma] \right)
\end{align*}
Now we will pick $\varsigma$ such that for $\varphi>0$, both $\frac{r}{2}-\varsigma > \varphi$ and $c_2 -  \varsigma > \varphi$. Thus,
\begin{equation*}
     P(A_\epsilon^c |(\bm{X_1},Y_1),\dots,(\bm{X_n},Y_n)) \leq \exp (-n\varphi) + \exp(-n\epsilon^2\varphi)
\end{equation*}
Hence, $P(A_\epsilon^c  |(\bm{X_1},Y_1), \dots, (\bm{X_n},Y_n)) \overset{p}{\to} 0$.
\qed

\subsection{Corollary \ref{corollary} Proof} \label{corollary_proof}
\autoref{thm-2} implies that $D_H(f_0,f) \overset{p}{\to} 0$ where $D_H(f_0,f)$ is the Hellinger distance between $f_0$ and $f$ as in \autoref{Hellinger_distance} and $f$ is a random draw from the posterior. Recall from \autoref{predictive density}, the predictive density function \begin{equation*}
    \hat{f}_n(.) = \int f(.)\hspace{1mm}\d P(f|(\bm{X_1},Y_1),\dots,(\bm{X_n},Y_n))
\end{equation*}
gives rise to the predictive conditional quantile function,  $\hat{\mu}_n(\bm{x}) = Q_{\tau,\hat{f}_n}(y |\bm{X}=\bm{x})$. We next show that $D_H(f_0,\hat{f}_n) \overset{p}{\to} 0$, which in turn implies $\hat{\mu}_n(\bm{x})$ converges in $L_1$-norm to the true conditional quantile function,
\begin{equation*}
\mu_0(\bm{x}) = Q_{\tau,f_0}(y |\bm{X}=\bm{x}) = \beta_0 + \sum_{j=1}^k \beta_j \frac{1}{1 + \exp \left( -\gamma_{j0}- \sum_{h=1}^p \gamma_{jh} x_{ih} \right)} 
\end{equation*} 
First we show that $D_H(f_0,\hat{f}_n) \overset{p}{\to} 0$. Let $X^n = ((\bm{X_1},Y_1),\dots,(\bm{X_n},Y_n))$. For any $\epsilon >0$:
\begin{align*}
    D_H(f_0,\hat{f}_n) & \leq \int D_H(f_0,f) \hspace{1mm} \d \pi_n(f|X^n) \\
    & \quad \text{By Jensen's Inequality} \\
    & \leq \int_{A_\epsilon} D_H(f_0,f) \hspace{1mm} \d \pi_n(f|X^n) + \int_{A_\epsilon^c} D_H(f_0,f) \hspace{1mm} \d \pi_n(f|X^n) \\
    & \leq \int_{A_\epsilon} \epsilon \hspace{1mm} \d \pi_n(f|X^n) + \int_{A_\epsilon^c} D_H(f_0,f) \hspace{1mm} \d \pi_n(f|X^n) \\
    & \leq \enskip \epsilon + \int_{A_\epsilon^c} D_H(f_0,f) \hspace{1mm} \d \pi_n(f|X^n)
\end{align*}
The second term goes to zero in probability by \autoref{thm-2} and $\epsilon$ is arbitrary, so $D_H(f_0,\hat{f}_n) \overset{p}{\to} 0$. 

In the remaining part of the proof, for notational simplicity, we take $\hat{\mu}_n(\bm{x})$ and $\mu_0(\bm{x})$ to be $\hat{\mu}$ and $\hat{\mu}_0$ respectively. The Hellinger distance between $f_0$ and $\hat{f}_n$ is  
\begin{align}
    D_H(f_0,\hat{f}_n) & = \left( \iint \left[ \sqrt{\hat{f}_n(\bm{x},y)} - \sqrt{f_0(\bm{x},y)} \right]^2 \hspace{1mm} \d y \hspace{1mm} \d x \right)^{1/2} \nonumber \\
    & = \left( \iint \tau(1-\tau) \left[ \exp \left(-\frac{1}{2}(y-\hat{\mu}_n) (\tau-I_{(y\leq \hat{\mu}_n)}) \right) - \exp \left(-\frac{1}{2}(y-\mu_0) (\tau-I_{(y\leq\mu_0)}) \right) \right]^2 \hspace{1mm} \d y \hspace{1mm} \d \bm{x} \right)^{1/2} \nonumber \\
    & = \left(2 - 2 \iint \tau(1-\tau) \exp \left(-\frac{1}{2}(y-\hat{\mu}_n) (\tau-I_{(y\leq \hat{\mu}_n)}) -\frac{1}{2}(y-\mu_0) (\tau-I_{(y\leq\mu_0)}) \right) \hspace{1mm} \d y \hspace{1mm} \d \bm{x} \right)^{1/2} \nonumber \\
    & \quad \text{let,} \enskip T = -\frac{1}{2}(y-\hat{\mu}_n) (\tau-I_{(y\leq \hat{\mu}_n)}) -\frac{1}{2}(y-\mu_0) (\tau-I_{(y\leq\mu_0)}) \nonumber \\
    & = \left(2 - 2 \iint \tau(1-\tau) \exp \left(T \right) \hspace{1mm} \d y \hspace{1mm} \d \bm{x} \right)^{1/2} \label{Hellinger distance predictive}
\end{align}
Now let's break $T$ into two cases: (a)\hspace{1mm}$\hat{\mu}_n \leq \mu_0$, and (b)\hspace{1mm}$\hat{\mu}_n > \mu_0$.
\vspace{-2mm}
\begin{enumerate}[label=Case-(\alph*),wide=0pt,leftmargin=.1in] 
\setlength\itemsep{1em}
    \item $\hat{\mu}_n \leq \mu_0$
        \begin{align*}
             T = \begin{cases} 
                 - \left(y - \frac{\hat{\mu}_n+\mu_0}{2} \right) \tau,  & \hat{\mu}_n \leq \mu_0 < y \\  
                 - \left(y - \frac{\hat{\mu}_n+\mu_0}{2} \right) \tau + \frac{(y-\mu_0)}{2}, & \hat{\mu}_n \leq \frac{\hat{\mu}_n+\mu_0}{2} < y \leq \mu_0 \\
                 - \left(y - \frac{\hat{\mu}_n+\mu_0}{2} \right) (\tau-1) - \frac{(y-\hat{\mu}_n)}{2},   & \hat{\mu}_n  < y \leq \frac{\hat{\mu}_n+\mu_0}{2} \leq \mu_0 \\
                 - \left(y - \frac{\hat{\mu}_n+\mu_0}{2} \right) (\tau-1), &  y \leq \hat{\mu}_n \leq \mu_0
                \end{cases} 
        \end{align*}
    \item $\hat{\mu}_n > \mu_0$ 
        \begin{align*}
             T = \begin{cases} 
                 - \left(y - \frac{\hat{\mu}_n+\mu_0}{2} \right) \tau,  & \mu_0 \leq \hat{\mu}_n < y \\  
                 - \left(y - \frac{\hat{\mu}_n+\mu_0}{2} \right) \tau + \frac{(y-\hat{\mu}_n)}{2}, &  \mu_0 \leq \frac{\hat{\mu}_n+\mu_0}{2} < y \leq \hat{\mu}_n \\
                 - \left(y - \frac{\hat{\mu}_n+\mu_0}{2} \right) (\tau-1) - \frac{(y-\mu_0)}{2}, & \mu_0 < y \leq \frac{\hat{\mu}_n+\mu_0}{2} \leq \hat{\mu}_n \\
                 - \left(y - \frac{\hat{\mu}_n+\mu_0}{2} \right) (\tau-1), &  y \leq \mu_0 \leq \hat{\mu}_n
                \end{cases} 
        \end{align*}    
\end{enumerate}
Hence now,
\begin{align*}
    & \int \tau(1-\tau) \exp \left(T \right) \hspace{1mm} \d y \\
    & = \int \left[I_{(\hat{\mu}_n \leq \mu_0)} + I_{( \hat{\mu}_n > \mu_0)} \right] \tau(1-\tau) \exp \left(T \right) \d y \\
    & = I_{(\hat{\mu}_n \leq \mu_0)} \tau(1-\tau) \times \left[ \int_{\mu_0}^\infty \exp \left\{- \left(y - \frac{\hat{\mu}_n+\mu_0}{2} \right) \tau \right\} \hspace{1mm} \d y + \int_{\frac{\hat{\mu}_n+\mu_0}{2}}^{\mu_0} \exp \left\{- \left(y - \frac{\hat{\mu}_n+\mu_0}{2} \right) \tau + \frac{(y-\mu_0)}{2} \right\} \hspace{1mm} \d y \right. \\
    & \quad \left. + \int_{\hat{\mu}_n}^{\frac{\hat{\mu}_n+\mu_0}{2}}\exp \left\{- \left(y - \frac{\hat{\mu}_n+\mu_0}{2} \right) (\tau-1) - \frac{(y-\hat{\mu}_n)}{2} \right\} \hspace{1mm} \d y + \int_{-\infty}^{\hat{\mu}_n} \exp \left\{- \left(y - \frac{\hat{\mu}_n+\mu_0}{2} \right) (\tau-1) \right\} \hspace{1mm} \d y \right] \\
    & \quad + I_{(\hat{\mu}_n > \mu_0)} \tau(1-\tau) \times \left[ \int_{\hat{\mu}_n}^\infty \exp \left\{- \left(y - \frac{\hat{\mu}_n+\mu_0}{2} \right) \tau \right\} \hspace{1mm} \d y + \int_{\frac{\hat{\mu}_n+\mu_0}{2}}^{\hat{\mu}_n} \exp \left\{- \left(y - \frac{\hat{\mu}_n+\mu_0}{2} \right) \tau +  \frac{(y-\hat{\mu}_n)}{2} \right\} \hspace{1mm} \d y \right. \\
    & \quad \left. + \int_{\mu_0}^{\frac{\hat{\mu}_n+\mu_0}{2}}\exp \left\{- \left(y - \frac{\hat{\mu}_n+\mu_0}{2} \right) (\tau-1) - \frac{(y-\mu_0)}{2} \right\} \hspace{1mm} \d y + \int_{-\infty}^{\mu_0} \exp \left\{- \left(y - \frac{\hat{\mu}_n+\mu_0}{2} \right) (\tau-1) \right\} \hspace{1mm} \d y \right] \\
    & = \frac{1-\tau}{1-2\tau}\exp \left(-\frac{\abs{\hat{\mu}_n-\mu_0}}{2} \tau \right) - \frac{\tau}{1-2\tau} \exp \left(-\frac{\abs{\hat{\mu}_n-\mu_0}}{2} (1-\tau) \right)
\end{align*}
Substituting the above expression in \autoref{Hellinger distance predictive} we get,
\begin{equation*}
    D_H(f_0,\hat{f}_n) = \left(2 - 2 \int \left[ \frac{1-\tau}{1-2\tau}\exp \left(-\frac{\abs{\hat{\mu}_n-\mu_0}}{2} \tau \right) - \frac{\tau}{1-2\tau} \exp \left(-\frac{\abs{\hat{\mu}_n-\mu_0}}{2} (1-\tau) \right) \right] \d \bm{x} \right)^{1/2}
\end{equation*}
Since $D_H(f_0,\hat{f}_n) \overset{p}{\to} 0$,
\begin{equation*}
    \int \left[ \frac{1-\tau}{1-2\tau}\exp \left(-\frac{\abs{\hat{\mu}_n-\mu_0}}{2} \tau \right) - \frac{\tau}{1-2\tau} \exp \left(-\frac{\abs{\hat{\mu}_n-\mu_0}}{2} (1-\tau) \right) \right] \d \bm{x} \overset{p}{\to} 1
\end{equation*}
Our next step is to show that above expression implies that $\abs{\hat{\mu}_n-\mu_0} \to 0$ a.s. on a set $\Omega$, with probability tending to 1, and hence $\int \abs{\hat{\mu}_n-\mu_0} \d \bm{x} \overset{p}{\to} 0$.

We are going to prove this using contradiction technique. Suppose that, $\abs{\hat{\mu}_n-\mu_0} \nrightarrow 0$ a.s. on $\Omega$. Then, there exists an $\epsilon > 0$ and a subsequence $\hat{\mu}_{n_i}$ such that $\abs{\hat{\mu}_{n_i}-\mu_0} > \epsilon$ on a set $A$ with $P(A)>0$. Now decompose the integral as
\begin{align*}
    \int & \left[ \frac{1-\tau}{1-2\tau}\exp \left(-\frac{\abs{\hat{\mu}_n-\mu_0}}{2} \tau \right) - \frac{\tau}{1-2\tau} \exp \left(-\frac{\abs{\hat{\mu}_n-\mu_0}}{2} (1-\tau) \right) \right] \d \bm{x} \\
    & = \int_A \left[ \frac{1-\tau}{1-2\tau}\exp \left(-\frac{\abs{\hat{\mu}_n-\mu_0}}{2} \tau \right) - \frac{\tau}{1-2\tau} \exp \left(-\frac{\abs{\hat{\mu}_n-\mu_0}}{2} (1-\tau) \right) \right] \d \bm{x} \\
    & \quad + \int_{A^c} \left[ \frac{1-\tau}{1-2\tau}\exp \left(-\frac{\abs{\hat{\mu}_n-\mu_0}}{2} \tau \right) - \frac{\tau}{1-2\tau} \exp \left(-\frac{\abs{\hat{\mu}_n-\mu_0}}{2} (1-\tau) \right) \right] \d \bm{x} \\
    & \leq \underbrace{P(A)}_{>0} \underbrace{\left[ \frac{(1-\tau)\exp(-\epsilon\tau/2) - \tau\exp(-\epsilon(1-\tau)/2)}{1-2\tau} \right]}_{<1 \enskip \text{(max $=1$ for $\epsilon=0$) and strictly $\downarrow$ for } \epsilon \in (0,\infty)} + \underbrace{P(A^c)}_{<1} \hspace{1mm} < 1
\end{align*}
So we have a contradiction since the integral converges in probability to 1. Thus $\abs{\hat{\mu}_n-\mu_0} \to 0$ a.s. on $\Omega$. Once we apply Scheffe's theorem we get $\int \abs{\hat{\mu}_n-\mu_0} \d \bm{x} \to 0$ a.s. on $\Omega$ and hence $\int \abs{\hat{\mu}_n-\mu_0} \d \bm{x} \overset{p}{\to} 0$. \qed


Below we prove the \autoref{thm-1} and for that we make use of \autoref{thm-2} and Corollary \ref{corollary}.

\subsection{\autoref{thm-1} Proof} \label{thm1_proof}
We proceed by showing that with $\F_n$ as in \autoref{sieve}, the prior $\pi_n$ of \autoref{thm-1} satisfies the condition (i) and (ii) of \autoref{thm-2}.

The proof of \autoref{thm-2} condition-(i) presented in \citet[proof of Theorem~1 on p. 639]{Lee-2000} holds in BQRNN case without any change. Next we need to show that condition-(ii) holds in BQRNN model. Let $K_\delta$ be the KL-neighborhood of the true density $f_0$ as in \autoref{KL neighborhood} and $\mu_0$ the corresponding conditional quantile function. We first fix a closely approximating neural network $\mu^*$ of $\mu_0$. We then find a neighborhood $M_\varsigma$ of $\mu^*$ as in \autoref{mdef} and show that this neighborhood has sufficiently large prior probability. Suppose that $\mu_0$ is continuous. For any $\delta >0$, choose $\epsilon = \delta/2$ in theorem from \citet[Theorem~1 on p.184]{Funahashi-1989} and let $\mu^*$ be a neural network such that $\underset{x\in \mathcal{X}}{\sup}\abs{\mu^*-\mu_0} < \epsilon$. Let $\varsigma=(\sqrt{\epsilon}/5n^a)=\sqrt{(\delta/50)}n^{-a}$ in Lemma \ref{lemma-5}. Then following derivation will show us that for any $\tilde{\mu}\in M_\varsigma, D_K(f_0,\tilde{f})\leq \delta$ i.e. $M_\varsigma \subset K_\delta$. 
\begin{align*}
    D_K(f_0,\tilde{f}) & = \iint f_0(x,y) \log \frac{f_0(x,y)}{\tilde{f}(x,y)} \hspace{1mm} \d y \hspace{1mm} \d x  \\
    & = \iint \left[(y-\tilde{\mu}) (\tau-I_{(y\leq\tilde{\mu})}) - (y-\mu_0)(\tau-I_{(y\leq \mu_0)}) \right] f_0(y|x) \hspace{1mm} f_0(x)\hspace{1mm} \d y \hspace{1mm} \d x  \\
    & \quad \text{let,} \enskip T = (y-\tilde{\mu}) (\tau-I_{(y\leq\tilde{\mu})}) - (y-\mu_0)(\tau-I_{(y\leq \mu_0)}) \\  
    & = \int \left[ \int T f_0(y|x) \hspace{1mm} \d y \right]f_0(x) \hspace{1mm} \d x 
\end{align*}
Now let's break $T$ into two cases: (a)\hspace{1mm}$\tilde{\mu} \geq \mu_0$, and (b)\hspace{1mm} $\tilde{\mu} < \mu_0$.
\vspace{-2mm}
\begin{enumerate}[label=Case-(\alph*),wide=0pt,leftmargin=.1in] 
\setlength\itemsep{1em}
    \item $\tilde{\mu} \geq \mu_0$
        \begin{align*}
             T = \begin{cases} 
                 (\mu_0-\tilde{\mu})\tau,  & \mu_0 \leq \tilde{\mu} < y \\  
                 (\mu_0-\tilde{\mu})\tau-(y-\tilde{\mu}), &  \mu_0 < y \leq \tilde{\mu} \\
                 (\mu_0-\tilde{\mu})(\tau-1), &  y \leq \mu_0 \leq \tilde{\mu}
                \end{cases} 
        \end{align*}
    \item $\tilde{\mu} \leq \mu_0$ 
        \begin{align*} 
             T = \begin{cases} 
                (\mu_0-\tilde{\mu})\tau,  & \tilde{\mu} \leq \mu_0 < y \\  
                (\mu_0-\tilde{\mu})(\tau-1)+(y-\tilde{\mu}), &  \tilde{\mu} < y \leq \mu_0 \\
                (\mu_0-\tilde{\mu})(\tau-1), &  y \leq \tilde{\mu} \leq \mu_0
                \end{cases} 
        \end{align*}
\end{enumerate}
So now,
\begin{align*}
    \int T f_0(y|x) \hspace{1mm} \d y & = \int \left[ I_{(\tilde{\mu} - \mu_0 \geq 0)} \times \left\{ (\tilde{\mu}-\mu_0)(1-\tau)I_{(y\leq \mu_0)} - (y-\tilde{\mu})I_{(\mu_0 < y \leq \tilde{\mu})} - (\tilde{\mu}-\mu_0)\tau I_{(y > \mu_0)} \right\} \right. \\
    & \hspace{6mm} \left. + I_{(\tilde{\mu} - \mu_0 < 0)} \times \left\{ (\tilde{\mu}-\mu_0)(1-\tau)I_{(y\leq \mu_0)} + (y-\tilde{\mu})I_{(\tilde{\mu} < y \leq \mu_0)} - (\tilde{\mu}-\mu_0)\tau I_{(y > \mu_0)} \right\} \right] f_0(y|x) \hspace{1mm} \d y \\
    & = \int \left[(\tilde{\mu}-\mu_0)(1-\tau)I_{(y\leq \mu_0)} - (\tilde{\mu}-\mu_0)\tau I_{(y > \mu_0)} \right.\\
    & \hspace{6mm} \left. - (y-\mu_0+\mu_0-\tilde{\mu})I_{(\mu_0 < y \leq \tilde{\mu})} + (y-\mu_0+\mu_0-\tilde{\mu})I_{(\tilde{\mu} < y \leq \mu_0)} \right] f_0(y|x) \hspace{1mm} \d y \\
    & \quad \text{let,} \enskip z = y - \mu_0, b=\tilde{\mu}-\mu_0 \enskip \text{and note that} \enskip P(y\leq \mu_0|x)=\tau, \text{ and } P(y>\mu_0|x)=1-\tau.\\
    & = E \left[-(z-b)I_{(0<z<b)} + (z-b)I_{(b<z<0)}|x \right] \\
    & \leq E \left[bI_{(0<z<b)} - bI_{(b<z<0)}|x \right] \\
    & = \abs{b} \times \left[P(0<z<b|x) + P(b<z<0|x) \right] \\
    & = \abs{b} \times P(0<\abs{z}<\abs{b}|x) \\
    & \leq \abs{b}
\end{align*}
Hence,
\begin{align*}
    \iint T f_0(y|x) \hspace{1mm} \d y \hspace{1mm} \d x & \leq \int \abs{b} f_0(x) \hspace{1mm} \d x \\
    & = \int \abs{\tilde{\mu}-\mu_0} f_0(x) \hspace{1mm} \d x \\
    & = \int \abs{\tilde{\mu}-\mu^*+\mu^*-\mu_0} f_0(x) \hspace{1mm} \d x \\
    & \leq \int \left[\underset{x\in \mathcal{X}}{\sup}\abs{\tilde{\mu}-\mu^*} + \underset{x\in \mathcal{X}}{\sup}\abs{\mu^*-\mu_0} \right] f_0(x) \hspace{1mm} \d x \\
     \enskip \text{Use Lemma \ref{lemma-5}  and}&\text{ \citet[Theorem~1 on p.184]{Funahashi-1989} to bound the first and second term respectively.} \\
    & \leq \int \left[\epsilon+\epsilon \right] f_0(x) \hspace{1mm} \d x \\
    & = 2\epsilon = \delta
\end{align*}
Finally we prove that $\forall \delta,\nu>0, \exists N_\nu$ s.t. $\pi_n(K_\delta) \geq \exp(-n\nu) \hspace{1mm} \forall n\geq N_\nu$,
\begin{align*}
    \pi_n(K_\delta) & \geq \pi_n(M_\varsigma) \\
    & = \prod_{i=1}^{\tilde{d}_n} \int_{\theta_i-\varsigma}^{\theta_i+\varsigma} \frac{1}{\sqrt{2\pi \sigma_0^2}} \exp \left(-\frac{1}{2\sigma_0^2} u^2 \right) \d u \\
    & \geq \prod_{i=1}^{\tilde{d}_n} 2\varsigma \underset{u\in [\theta_i-1,\theta_i+1]}{\inf} \frac{1}{\sqrt{2\pi \sigma_0^2}} \exp \left(-\frac{1}{2\sigma_0^2} u^2 \right) \\
    & = \prod_{i=1}^{\tilde{d}_n} \varsigma \sqrt{\frac{2}{\pi \sigma_0^2}}\exp \left(-\frac{1}{2\sigma_0^2} \vartheta_i \right) \\
    & \quad \vartheta_i = \max((\theta_i-1)^2,(\theta_i+1)^2) \\
    & \geq \left(\varsigma  \sqrt{\frac{2}{\pi \sigma_0^2}} \right)^{\tilde{d}_n} \exp \left( -\frac{1}{2\sigma_0^2} \vartheta\tilde{d}_n \right) \qquad \text{where, } \vartheta = \underset{i}{\max}(\vartheta_1,\dots,\vartheta_{\tilde{d}_n}) \\
    & = \exp \left(-\tilde{d}_n \left[ a\log n - \log \sqrt{\frac{\delta}{25\pi\sigma_0^2}} \right] -\frac{1}{2\sigma_0^2} \vartheta\tilde{d}_n \right) \\
    & \quad \varsigma = \sqrt{\frac{\delta}{50}} n^{-a} \\
    & \geq \exp\left( -\left[ 2a \log n + \frac{\vartheta }{2\sigma_0^2} \right] \tilde{d}_n \right) \qquad \text{for large } n \\
    & \geq \exp\left( -\left[ 2a \log n + \frac{\vartheta }{2\sigma_0^2} \right] (p+3)n^a \right) \\
    & \quad \tilde{d}_n = (p+2)\tilde{k}_n + 1 \leq (p+3)n^a \\
    & \geq \exp (-n \nu) \qquad \text{for any } \nu \text{ and } \forall n \geq N_\nu \text{ for some } N_\nu
\end{align*}
Hence, we have proved that both the conditions of \autoref{thm-2} hold. The result of \autoref{thm-1} thereby follows from the Corollary \ref{corollary} which is derived from \autoref{thm-2}.


\end{document}